\documentclass[10pt]{article}

\usepackage[margin=1in]{geometry}
\usepackage{amsmath,amssymb,amsthm,mathtools,mathrsfs}
\usepackage{enumitem}
\usepackage{array,booktabs}
\usepackage{aliascnt}
\usepackage[colorlinks=true,linkcolor=blue,citecolor=blue,urlcolor=blue]{hyperref}
\usepackage[nameinlink,capitalise,noabbrev]{cleveref}
\usepackage{microtype}

\numberwithin{equation}{section}
\setlist[enumerate]{leftmargin=2.2em,itemsep=0.25em,topsep=0.4em}
\setlist[itemize]{leftmargin=2.2em,itemsep=0.25em,topsep=0.4em}

\newtheorem{theorem}{Theorem}[section]
\newaliascnt{proposition}{theorem}
\newtheorem{proposition}[proposition]{Proposition}
\aliascntresetthe{proposition}
\newaliascnt{lemma}{theorem}
\newtheorem{lemma}[lemma]{Lemma}
\aliascntresetthe{lemma}
\newaliascnt{corollary}{theorem}
\newtheorem{corollary}[corollary]{Corollary}
\aliascntresetthe{corollary}
\newaliascnt{assumption}{theorem}
\newtheorem{assumption}[assumption]{Assumption}
\aliascntresetthe{assumption}

\theoremstyle{definition}
\newaliascnt{definition}{theorem}

\aliascntresetthe{definition}
\newtheorem*{definition*}{Definition}
\newaliascnt{example}{theorem}

\aliascntresetthe{example}

\theoremstyle{remark}
\newaliascnt{remark}{theorem}

\aliascntresetthe{remark}

\crefname{theorem}{Theorem}{Theorems}
\Crefname{theorem}{Theorem}{Theorems}
\crefname{proposition}{Proposition}{Propositions}
\Crefname{proposition}{Proposition}{Propositions}
\crefname{lemma}{Lemma}{Lemmas}
\Crefname{lemma}{Lemma}{Lemmas}
\crefname{corollary}{Corollary}{Corollaries}
\Crefname{corollary}{Corollary}{Corollaries}
\crefname{definition}{Definition}{Definitions}
\Crefname{definition}{Definition}{Definitions}
\crefname{assumption}{Assumption}{Assumptions}
\Crefname{assumption}{Assumption}{Assumptions}
\crefname{example}{Example}{Examples}
\Crefname{example}{Example}{Examples}
\crefname{remark}{Remark}{Remarks}
\Crefname{remark}{Remark}{Remarks}

\newcommand{\R}{\mathbb R}
\newcommand{\C}{\mathbb C}
\newcommand{\Ex}{\mathbb E}
\newcommand{\Prob}{\mathbb P}

\newcommand{\cD}{\mathcal D}
\newcommand{\cE}{\mathcal E}
\newcommand{\cF}{\mathcal F}
\newcommand{\Pres}{\mathrm{Pres}}
\newcommand{\Cyl}{\mathrm{Cyl}}
\newcommand{\Dom}{\operatorname{Dom}}
\newcommand{\Div}{\operatorname{div}}
\newcommand{\Span}{\operatorname{span}}
\newcommand{\supp}{\operatorname{supp}}
\newcommand{\Geo}{\operatorname{Geo}}
\newcommand{\len}{\operatorname{len}}

\newcommand{\dd}{\,\mathrm d}

\hypersetup{
  pdftitle={Closed Response Calculus and Joint Densities for the Liouville Quantum Gravity Metric},
  pdfauthor={Chunhao Cai},
  pdfsubject={Closed response calculus, geodesic occupations, and joint densities for the LQG metric},
  pdfkeywords={Liouville quantum gravity metric, Gaussian free field, response calculus, geodesic occupation measure, Dirichlet form, energy-image-density absolute continuity, joint density}
}

\title{Closed Response Calculus and Joint Densities\\
for the Liouville Quantum Gravity Metric}
\author{Chunhao Cai\\
School of Mathematics (Zhuhai), Sun Yat-Sen University\\
\texttt{caichh9@mail.sysu.edu.cn}}
\date{August 2026}

\begin{document}

\maketitle

\begin{abstract}
We give a general criterion for transferring first-order responses on an
underlying probability space to a closed differential calculus on the law of
an observable.  An integration-by-parts identity removes presentation
ambiguity and yields a closable gradient, its divergence, and a closed Markov
form.

We apply this scheme to the subcritical Liouville quantum gravity metric
throughout the range $0<\gamma<2$.
Sequential compactness of Weyl-perturbed geodesics identifies the derivative
of a logarithmic distance ratio with the Sobolev Riesz representative of the
difference of two normalized geodesic occupation measures.  The induced form
on the projective metric law is the image of a Gaussian directional form and
has the energy-image-density property.  Finally, fixed-target confluence and
a leaf-elimination argument make the response Gram matrix positive definite
for every finite forest of marked pairs.  The corresponding vector of
logarithmic distance ratios therefore has a Lebesgue density.
\end{abstract}



\begingroup
\small
\tableofcontents
\endgroup

\section{Introduction}
\label{sec:intro}

Variational observables in random geometry are often defined by
minimizing over paths or other competitors.  Their first-order response is
therefore determined by the minimizers that are active at the unperturbed
point.  For a random metric, the relevant response is the normalized
occupation measure of a geodesic.  Passing from this pathwise derivative to a
differential operator on the law of the observable raises two questions:
different cylinder presentations of the same image function must have the
same gradient, and the resulting gradient must be closable.

We separate the abstract image-law argument from its LQG verification.
Once measurable linear response sections are available, conditional
integration by parts removes the presentation ambiguity and yields a closed
calculus on the image law.  For the LQG metric, the response sections are
constructed directly from sequential compactness of Weyl-perturbed geodesics.
The resulting Gaussian directional form satisfies the corresponding
energy-image-density absolute-continuity statement, and geodesic confluence
gives joint densities for forest-indexed logarithmic distance ratios.

The proof of joint density has a short conceptual backbone.  The
energy-image-density statement turns almost-sure positivity of the response
Gram determinant into absolute continuity, while fixed-target confluence and
leaf elimination provide this positivity for forests.

\subsection{Image integration by parts, divergence, and closure}

Let $(\Omega,\cF,\Prob)$ be a probability space and $(X,\mathcal X)$ a
measurable space.  Let
\[
 J:(\Omega,\cF,\Prob)\longrightarrow(X,\mathcal X)
\]
be measurable.  For a measurable map $F$, write
\[
 F_\#\nu:=\nu\circ F^{-1}
\]
for the pushforward of a measure $\nu$.  Set
\[
 \mu:=J_\#\Prob,
 \qquad
 \mu(A)=\Prob(J\in A),\quad A\in\mathcal X,
\]
for the law of $J$.

Let $I$ be countable, and let $r_a:X\to\R$, $a\in I$, be measurable
coordinates.  Let $V$ be a real vector space of deterministic directions and
$V_0\subset V$ a countable $\mathbb Q$-linear subspace.  We work with a
measurable field of separable Hilbert spaces $(H_x)_{x\in X}$.  Each $v\in V$
determines a measurable section $x\mapsto v_x\in H_x$, linearly in $v$, and
each $a\in I$ has a measurable response section $x\mapsto g_a(x)\in H_x$.
The prescribed first-order response is
\[
 \partial_v r_a(x):=\langle g_a(x),v_x\rangle_{H_x},
 \qquad a\in I,\quad v\in V.
\]
The construction of these response sections is model dependent; in the LQG
application it is carried out directly in \cref{sec:lqg}.
We identify $v$ with the section $x\mapsto v_x$ when no confusion can arise,
and write
\[
 L^2_\mu(H):=L^2(X,\mu;H_x),
 \qquad
 \|u\|_{L^2_\mu(H)}^2
 :=\int_X\|u(x)\|_{H_x}^2\,\dd\mu(x).
\]

For $n\ge1$, put
\[
 C_b^\infty(\R^n)
 :=
 \bigl\{\varphi\in C^\infty(\R^n):
 \|\partial^\alpha\varphi\|_\infty<\infty
 \text{ for every multi-index }\alpha\bigr\}.
\]
The set of formal cylinder presentations is
\[
 \Pres
 :=
 \bigl\{(\varphi;a_1,\ldots,a_n):
 n\ge1,\ \varphi\in C_b^\infty(\R^n),\
 (a_1,\ldots,a_n)\in I^n\bigr\}.
\]
For $P=(\varphi;a_1,\ldots,a_n)\in\Pres$, define
\[
 \begin{aligned}
 F_P(x)
 &:=\varphi\bigl(r_{a_1}(x),\ldots,r_{a_n}(x)\bigr),\\
 G_P(x)
 &:=\sum_{i=1}^n
 \partial_i\varphi\bigl(r_{a_1}(x),\ldots,r_{a_n}(x)\bigr)g_{a_i}(x).
 \end{aligned}
\]
Then
\[
 \begin{aligned}
 \partial_vF_P(x)
 &:=\sum_{i=1}^n
 \partial_i\varphi\bigl(r_{a_1}(x),\ldots,r_{a_n}(x)\bigr)
 \partial_vr_{a_i}(x)\\
 &=\langle G_P(x),v_x\rangle_{H_x},
 \qquad v\in V.
 \end{aligned}
\]
Set \(\Cyl:=\{F_P:P\in\Pres\}\).
A cylinder function may have more than one presentation.  At this stage,
$F_P=F_Q$ $\mu$-almost everywhere does not yet imply $G_P=G_Q$ in
$L^2_\mu(H)$.

For $P\in\Pres$ and $v\in V_0$, put
\[
 (F_Pv)(x):=F_P(x)v_x,
 \qquad
 U_0:=\Span\{F_Pv:P\in\Pres,\ v\in V_0\}.
\]
On the original probability space, set
\[
 Y_P:=F_P\circ J.
\]
For $v\in V_0$, let $\partial_vY_P$ denote the prescribed directional
derivative of $Y_P$ on $\Omega$.

An \emph{ambient score} for $v\in V_0$ is a random variable
$\beta_v\in L^2(\Prob)$ such that
\[
 \Ex[\partial_vY_P]=\Ex[Y_P\beta_v],
 \qquad P\in\Pres.
\]
Given such a $\beta_v$, a measurable $b_v:X\to\R$ is a
\emph{conditional image score} if
\[
 b_v(J)=\Ex[\beta_v\mid\sigma(J)]
 \qquad\Prob\text{-a.s.}
\]
When $X$ is standard Borel, such a factor exists by the Doob--Dynkin
lemma; for a general measurable image space its existence is part of the
model-dependent input.

\begin{assumption}
\label{ass:intro-response}
The following hold.
\begin{enumerate}[label=\textnormal{(\roman*)}]
\item For $a\in I$, $v\in V_0$, and $P\in\Pres$,
\[
 g_a\in L^2_\mu(H),\qquad
 v\in L^2_\mu(H),\qquad
 \partial_vF_P\in L^2(\mu).
\]
\item The scalar cylinder algebra is dense:
\[
 \overline{\Cyl}^{\,L^2(\mu)}=L^2(\mu),
 \qquad
 \overline{U_0}^{\,L^2_\mu(H)}=L^2_\mu(H).
\]
\item Every $v\in V_0$ admits an ambient score $\beta_v$ and a conditional
image score $b_v$, and
\[
 \partial_vY_P=(\partial_vF_P)(J)
 \qquad\Prob\text{-a.s.},\quad P\in\Pres.
\]
\end{enumerate}
\end{assumption}

Conditional Jensen gives
\[
 \|b_v\|_{L^2(\mu)}^2
 =\Ex\bigl[|b_v(J)|^2\bigr]
 \le \Ex\bigl[|\beta_v|^2\bigr],
 \qquad v\in V_0,
\]
so $b_v\in L^2(\mu)$.  For
$U=\sum_{j=1}^kF_{P_j}v_j\in U_0$, set
\begin{equation}
 \Div_0U
 :=\sum_{j=1}^k
 \bigl(\partial_{v_j}F_{P_j}-F_{P_j}b_{v_j}\bigr).
 \label{eq:intro-divergence-core}
\end{equation}
A priori, this value may depend on the displayed representation of $U$.

\begin{theorem}
\label{thm:intro-image-closure}
Under \cref{ass:intro-response}, the following hold.
\begin{enumerate}[label=\textnormal{(\roman*)}]
\item If
\[
 U=\sum_{j=1}^kF_{P_j}v_j
  =\sum_{\ell=1}^rF_{Q_\ell}w_\ell
 \quad\text{in }L^2_\mu(H),
\]
then
\[
 \sum_{j=1}^k
 \bigl(\partial_{v_j}F_{P_j}-F_{P_j}b_{v_j}\bigr)
 =\sum_{\ell=1}^r
 \bigl(\partial_{w_\ell}F_{Q_\ell}-F_{Q_\ell}b_{w_\ell}\bigr)
 \quad\text{in }L^2(\mu).
\]
Hence \eqref{eq:intro-divergence-core} defines a linear map
\[
 \Div_0:U_0\longrightarrow L^2(\mu)
\]
satisfying
\begin{equation}
 \int_X\langle G_P(x),U(x)\rangle_{H_x}\,\dd\mu(x)
 =-\int_XF_P(x)\Div_0U(x)\,\dd\mu(x),
 \qquad P\in\Pres,\quad U\in U_0.
 \label{eq:intro-core-ibp}
\end{equation}

\item If $F_P=F_Q$ in $L^2(\mu)$, then $G_P=G_Q$ in
$L^2_\mu(H)$.  Thus the rule $D_0F_P:=G_P$ defines a densely defined
closable operator
\[
 D_0:\Cyl\subset L^2(\mu)\longrightarrow L^2_\mu(H).
\]
If $D=\overline{D_0}$, then
\[
 U_0\subset\Dom(D^*),
 \qquad
 D^*U=-\Div_0U\quad(U\in U_0).
\]

\item The bilinear form
\[
 \cE(F,G):=\int_X\langle DF(x),DG(x)\rangle_{H_x}\,\dd\mu(x),
 \qquad F,G\in\Dom(D),
\]
is a densely defined closed symmetric Markov form on $L^2(\mu)$.
\end{enumerate}
\end{theorem}

\subsection{Application to the LQG metric}
\label{sec:intro-lqg-application}

We now specialize the preceding construction to the LQG metric.  The
pathwise response used below is proved in \cref{sec:lqg} from sequential
compactness of Weyl-perturbed geodesics.

Fix $0<\gamma<2$.  Let $h$, defined on $(\Omega,\cF,\Prob)$, be a whole-plane
Gaussian free field modulo additive constants, and let $D_h$ be a fixed
normalization of its $\gamma$-LQG metric.  Put
\[
 \xi:=\frac{\gamma}{d_\gamma},
\]
where $d_\gamma$ is the LQG metric dimension exponent
\cite[(1.5)]{GwynneMiller2021}.  Weyl scaling gives
$D_{h+c}=e^{\xi c}D_h$ for constants $c$, so the distance ratios below are
independent of the representative of $h$.

Let $Q\subset\C$ be countable and dense, and set
\[
 I:=\bigl\{\{z,w\}:z,w\in Q,\ z\ne w\bigr\}.
\]
For $e=\{z,w\}\in I$, write $D_h(e):=D_h(z,w)$.  Fix $e_0\in I$ and define
\[
 R_e(h):=\log\frac{D_h(e)}{D_h(e_0)},
 \qquad e\in I.
\]

For finite $E\subset I$, let
\[
 G_E:=\left(\bigcup_{e\in E}e,E\right)
\]
be the graph determined by $E$.  We call $E$ a forest if $G_E$ is acyclic.
Equivalently,
\[
 |E'|<\left|\bigcup_{e\in E'}e\right|
 \qquad(\varnothing\ne E'\subseteq E).
\]

We denote by $\mathcal L^m$ Lebesgue measure on $\R^m$.  The density
statement below holds throughout the subcritical range $0<\gamma<2$.

\begin{theorem}
\label{thm:intro-forest}
Let $m\ge1$ and let
$e_0,e_1,\ldots,e_m\in I$ be pairwise distinct.  Set
\[
 E_F:=\{e_0,e_1,\ldots,e_m\},
 \qquad
 N:=\bigcup_{e\in E_F}e.
\]
Suppose that the finite graph $(N,E_F)$ is acyclic; equivalently,
\[
 |E'|<\left|\bigcup_{e\in E'}e\right|
 \qquad\text{for every }\varnothing\ne E'\subseteq E_F.
\]
Then
\[
 \bigl(R_{e_1}(h),\ldots,R_{e_m}(h)\bigr)_\#\Prob
 \ll \mathcal L^m.
\]
\end{theorem}

Since both \(Q\) and the reference edge may be chosen arbitrarily, the same
conclusion holds for any fixed deterministic finite forest of pairs in \(\C\),
after choosing \(Q\) to contain its vertices.  The forest condition is sufficient;
no converse is asserted for cyclic graphs.

The rest of the paper is organized as follows.  \Cref{sec:closure} proves
\cref{thm:intro-image-closure} from \cref{ass:intro-response}.
\Cref{sec:lqg} derives the LQG Weyl response, realizes the response gradients
in a Sobolev Hilbert space, verifies \cref{ass:intro-response}, establishes the required
energy-image-density absolute continuity, and proves \cref{thm:intro-forest} by forest nondegeneracy.

\section{Image calculus and closure}
\label{sec:closure}

Retain \(X,\mu\), the response sections, and the cylinder presentations
from \cref{sec:intro}, and let \(\mathcal X^\mu\) be the
\(\mu\)-completion of \(\mathcal X\).  We first record the elementary
presentation algebra used below.  If
\[
 P=(\varphi;a_1,\ldots,a_n),\qquad
 Q=(\psi;b_1,\ldots,b_m),
\]
then linear combinations and products are represented by concatenating the
coordinate lists and using the corresponding linear combination or product of
the defining functions.  Thus, at the presentation level,
\[
 \begin{aligned}
 F_{\alpha P+\beta Q}&=\alpha F_P+\beta F_Q,
 &G_{\alpha P+\beta Q}&=\alpha G_P+\beta G_Q,\\
 F_{PQ}&=F_PF_Q,
 &G_{PQ}&=F_PG_Q+F_QG_P.
 \end{aligned}
\]
More generally, smooth compositions satisfy the usual finite-dimensional
chain rule.  These identities are purely presentation-level statements; the
descent to functions in \(L^2(\mu)\) is proved below.

\subsection{Cylinder density}

For a sub-$\sigma$-field $\mathcal G\subseteq\mathcal X$, write
$\mathcal G^\mu$ for its $\mu$-completion.  We first record a criterion for
cylinder density.

Recall from \cref{sec:intro} that $I$ is the countable index set of the
coordinate family $(r_a)_{a\in I}$.

\begin{proposition}
\label{prop:projective-density}
Choose finite sets
\[
 I_1\subset I_2\subset\cdots\subset I,
 \qquad \bigcup_{n\ge1}I_n=I.
\]
For each $n$, write $I_n=\{a_{n,1},\ldots,a_{n,m_n}\}$ and define
\[
 \mathbf r_n(x):=
 \bigl(r_{a_{n,1}}(x),\ldots,r_{a_{n,m_n}}(x)\bigr),
 \qquad x\in X.
\]
If
\begin{equation}
 \mathcal X^\mu=\sigma(r_a:a\in I)^\mu,
\label{eq:sigma-generation}
\end{equation}
then
\[
 \overline{\bigcup_{n\ge1}
 \{\varphi(\mathbf r_n):\varphi\in C_c^\infty(\R^{m_n})\}}
 ^{\,L^2(\mu)}
 =L^2(\mu).
\]
\end{proposition}

\begin{proof}
Put $\mu_n=(\mathbf r_n)_\#\mu$.  Since $\mu_n$ is a probability measure on
$\R^{m_n}$, $C_c^\infty(\R^{m_n})$ is dense in $L^2(\mu_n)$.  Moreover,
\[
 \sigma(\mathbf r_1)\subset\sigma(\mathbf r_2)\subset\cdots,
 \qquad
 \bigvee_{n\ge1}\sigma(\mathbf r_n)=\sigma(r_a:a\in I).
\]
Hence martingale convergence and \eqref{eq:sigma-generation} give, for every
$F\in L^2(\mu)$,
\[
 \Ex_\mu[F\mid\sigma(\mathbf r_n)]\longrightarrow F
 \quad\text{in }L^2(\mu).
\]
Given $\varepsilon>0$, choose $n$ so that the preceding error is less than
$\varepsilon/2$.  By the Doob--Dynkin lemma,
\[
 \Ex_\mu[F\mid\sigma(\mathbf r_n)]=\psi_n(\mathbf r_n)
 \quad\mu\text{-a.e.}
\]
for some $\psi_n\in L^2(\mu_n)$.  Choose
$\varphi\in C_c^\infty(\R^{m_n})$ with
$\|\psi_n-\varphi\|_{L^2(\mu_n)}<\varepsilon/2$.  Then
\[
 \begin{aligned}
 \|F-\varphi(\mathbf r_n)\|_{L^2(\mu)}
 &\le
 \bigl\|F-\Ex_\mu[F\mid\sigma(\mathbf r_n)]\bigr\|_{L^2(\mu)}
 +\|\psi_n-\varphi\|_{L^2(\mu_n)}\\
 &<\varepsilon.
 \end{aligned}
\]
\end{proof}

Since $C_c^\infty(\R^{m_n})\subset C_b^\infty(\R^{m_n})$, the preceding
functions belong to $\Cyl$; in particular,
$\overline{\Cyl}^{\,L^2(\mu)}=L^2(\mu)$.  Density alone, however, does not
identify gradients attached to different presentations.

\subsection{Image divergence, presentation descent, and closure}

We now work under \cref{ass:intro-response}.  The first step is to derive on
the image space the integration-by-parts identity already anticipated in
\cref{thm:intro-image-closure}.

For $v\in V_0$ and $P\in\Pres$, pushforward by $J$, the response
compatibility in \cref{ass:intro-response}(iii), the ambient score identity,
and conditional expectation give
\begin{equation}
\begin{aligned}
 \int_X\partial_vF_P\,\dd\mu
 &=\Ex\bigl[(\partial_vF_P)(J)\bigr]
  =\Ex[\partial_vY_P]
  =\Ex[Y_P\beta_v]\\
 &=\Ex\!\left[F_P(J)\Ex[\beta_v\mid\sigma(J)]\right]
  =\int_XF_Pb_v\,\dd\mu.
\end{aligned}
\label{eq:image-log-derivative}
\end{equation}

\begin{proposition}
\label{prop:log-derivative}
Under \cref{ass:intro-response}, the expression in
\eqref{eq:intro-divergence-core} depends only on $U\in U_0$.  Hence it defines
a linear map $\Div_0:U_0\to L^2(\mu)$, and the integration-by-parts identity
\eqref{eq:intro-core-ibp} holds.
\end{proposition}

\begin{proof}
Fix $P,Q\in\Pres$ and $v\in V_0$.  Apply
\eqref{eq:image-log-derivative} to the product presentation $PQ$.  The
presentation-level product rule gives
\[
\begin{aligned}
 \int_XF_P\partial_vF_Q\,\dd\mu
 &=\int_X\partial_vF_{PQ}\,\dd\mu
   -\int_XF_Q\partial_vF_P\,\dd\mu\\
 &=-\int_XF_Q\bigl(\partial_vF_P-F_Pb_v\bigr)\,\dd\mu.
\end{aligned}
\]
Therefore
\begin{equation}
 \langle G_Q,F_Pv\rangle_{L^2_\mu(H)}
 =-\langle F_Q,\partial_vF_P-F_Pb_v\rangle_{L^2(\mu)}.
\label{eq:ibp-simple-field}
\end{equation}
For a displayed representation
$U=\sum_{j=1}^kF_{P_j}v_j$, let
\[
 S_U:=\sum_{j=1}^k
 \bigl(\partial_{v_j}F_{P_j}-F_{P_j}b_{v_j}\bigr).
\]
Summing \eqref{eq:ibp-simple-field} gives
\[
 \langle G_Q,U\rangle_{L^2_\mu(H)}
 =-\langle F_Q,S_U\rangle_{L^2(\mu)},
 \qquad Q\in\Pres.
\]
If $U=0$, the density of $\Cyl$ from \cref{ass:intro-response}(ii) gives
$S_U=0$.  Applying this to the difference of two displayed representations
shows that $S_U$ depends only on $U$.  This proves the first assertion, and
the last display is exactly \eqref{eq:intro-core-ibp}.
\end{proof}

\begin{proposition}
\label{prop:closure}
Under \cref{ass:intro-response}, for every $P,Q\in\Pres$,
\[
 \|F_P-F_Q\|_{L^2(\mu)}=0
 \quad\Longrightarrow\quad
 \|G_P-G_Q\|_{L^2_\mu(H)}=0.
\]
Consequently, the assignment
\begin{equation}
 D_0F_P:=G_P
\label{eq:descended-gradient}
\end{equation}
defines a densely defined closable operator
\[
 D_0:\Cyl\subset L^2(\mu)\longrightarrow L^2_\mu(H).
\]
If $D:=\overline{D_0}$, then
\begin{equation}
 U_0\subset\Dom(D^*),\qquad
 D^*U=-\Div_0U\quad(U\in U_0).
\label{eq:adjoint-on-core}
\end{equation}
\end{proposition}

\begin{proof}
First suppose that $\|F_P\|_{L^2(\mu)}=0$.  By
\cref{prop:log-derivative},
\[
 \langle G_P,U\rangle_{L^2_\mu(H)}=0,
 \qquad U\in U_0.
\]
The density of $U_0$ in \cref{ass:intro-response}(ii) therefore gives
$\|G_P\|_{L^2_\mu(H)}=0$.  Applying the same argument to the difference of
two presentations yields
\[
 \|F_P-F_Q\|_{L^2(\mu)}=0
 \quad\Longrightarrow\quad
 \|G_P-G_Q\|_{L^2_\mu(H)}=0.
\]
Hence \eqref{eq:descended-gradient} is well defined on $\Cyl$.  Its domain is
dense by \cref{ass:intro-response}(ii).

For closability, let $F_n\in\Cyl$ satisfy
\[
 F_n\longrightarrow0\quad\text{in }L^2(\mu),\qquad
 D_0F_n\longrightarrow Z\quad\text{in }L^2_\mu(H).
\]
For $U\in U_0$, \cref{prop:log-derivative} gives
\[
 \langle Z,U\rangle_{L^2_\mu(H)}
 =-\lim_{n\to\infty}\langle F_n,\Div_0U\rangle_{L^2(\mu)}=0.
\]
Again density of $U_0$ yields $Z=0$.

Finally, for $F\in\Cyl$ and $U\in U_0$,
\[
 \langle D_0F,U\rangle_{L^2_\mu(H)}
 =\langle F,-\Div_0U\rangle_{L^2(\mu)}.
\]
Thus $U\in\Dom(D_0^*)$ and $D_0^*U=-\Div_0U$.  Since
$(\overline{D_0})^*=D_0^*$, \eqref{eq:adjoint-on-core} follows.
\end{proof}

We next pass the presentation-level chain and product rules to the closure.  Here
$C_b^1(\R)$ denotes the space of $C^1$ functions whose value and first
derivative are bounded.

\begin{proposition}
\label{prop:closed-calculus}
For $D=\overline{D_0}$, the following hold.
\begin{enumerate}[label=\textnormal{(\roman*)}]
\item If $F\in\Dom(D)$ and $\theta\in C_b^1(\R)$ has uniformly continuous
      derivative, then $\theta(F)\in\Dom(D)$ and
      \begin{equation}
       D\theta(F)=\theta'(F)DF.
      \label{eq:closed-chain}
      \end{equation}
\item If $F,G\in\Dom(D)\cap L^\infty(\mu)$, then $FG\in\Dom(D)$ and
      \begin{equation}
       D(FG)=F\,DG+G\,DF.
      \label{eq:closed-product}
      \end{equation}
\end{enumerate}
\end{proposition}

\begin{proof}
Let $F\in\Cyl$.  Since $\theta'$ is uniformly continuous and
$\theta$ is Lipschitz, mollification gives $\theta_m\in C_b^\infty(\R)$ with
\[
 \|\theta_m-\theta\|_\infty
 +\|\theta_m'-\theta'\|_\infty\longrightarrow0.
\]
The presentation-level chain rule gives
\[
 D_0\theta_m(F)=\theta_m'(F)D_0F.
\]
Hence
\[
 \theta_m(F)\longrightarrow\theta(F)\quad\text{in }L^2(\mu),
 \qquad
 D_0\theta_m(F)\longrightarrow\theta'(F)D_0F
 \quad\text{in }L^2_\mu(H),
\]
and closedness of $D$ gives \eqref{eq:closed-chain} for $F\in\Cyl$.

Now let $F_n\in\Cyl$ converge to $F\in\Dom(D)$ in the graph norm of $D$.
Passing to a subsequence, assume $F_n\to F$ $\mu$-almost everywhere.  Since
$\theta$ is Lipschitz,
\[
 \|\theta(F_n)-\theta(F)\|_{L^2(\mu)}
 \le \|\theta'\|_\infty\|F_n-F\|_{L^2(\mu)}\longrightarrow0.
\]
Moreover,
\[
 \begin{aligned}
 D\theta(F_n)-\theta'(F)DF
 &=\theta'(F_n)(D_0F_n-DF)\\
 &\quad +(\theta'(F_n)-\theta'(F))DF.
 \end{aligned}
\]
The first term tends to zero in $L^2_\mu(H)$ because $\theta'$ is bounded,
and the second by dominated convergence.  Closedness proves
\eqref{eq:closed-chain}.

For the product rule, choose
\[
 M>\max\{\|F\|_\infty,\|G\|_\infty\}
\]
and $\tau\in C_b^\infty(\R)$ with $\tau(t)=t$ on $[-M,M]$.  Choose
$A_n,B_n\in\Cyl$ converging to $F,G$, respectively, in graph norm, and pass
to a common subsequence such that $A_n\to F$ and $B_n\to G$ almost
everywhere.  Put
\[
 F_n:=\tau(A_n),\qquad G_n:=\tau(B_n).
\]
Then $F_n,G_n\in\Cyl$ and
\[
 \|F_n\|_\infty,\|G_n\|_\infty\le\|\tau\|_\infty.
\]
The chain rule on cylinders gives
\[
 D_0F_n=\tau'(A_n)D_0A_n,
 \qquad
 D_0G_n=\tau'(B_n)D_0B_n.
\]
Since $\tau(F)=F$, $\tau'(F)=1$, and likewise for $G$,
\[
 F_n\longrightarrow F,
 \qquad D_0F_n\longrightarrow DF,
 \qquad
 G_n\longrightarrow G,
 \qquad D_0G_n\longrightarrow DG
\]
in the corresponding $L^2$ spaces.  For instance,
\[
 D_0F_n-DF
 =\tau'(A_n)(D_0A_n-DF)+(\tau'(A_n)-1)DF,
\]
where the first term tends to zero by graph convergence and the second by
dominated convergence.

The uniform bounds imply
\[
 F_nG_n\longrightarrow FG\quad\text{in }L^2(\mu).
\]
Using the presentation-level product rule,
\[
 \begin{aligned}
 D_0(F_nG_n)-(F\,DG+G\,DF)
 &=F_n(D_0G_n-DG)+(F_n-F)DG\\
 &\quad+G_n(D_0F_n-DF)+(G_n-G)DF.
 \end{aligned}
\]
The first and third terms tend to zero by the uniform bounds and graph
convergence; the second and fourth by dominated convergence.  Closedness of
$D$ gives \eqref{eq:closed-product}.
\end{proof}

A symmetric nonnegative form $(\cE,\mathcal D)$ on $L^2(\mu)$ is called
\emph{Markovian} if, for every normal contraction $C:\R\to\R$, that is,
\[
 C(0)=0,\qquad |C(s)-C(t)|\le |s-t|\quad(s,t\in\R),
\]
and every $F\in\mathcal D$,
\[
 C(F)\in\mathcal D,\qquad
 \cE(C(F),C(F))\le \cE(F,F).
\]

\begin{corollary}
\label{cor:closed-energy}
For $D=\overline{D_0}$, define
\begin{equation}
 \cE(F,G)
 :=\int_X\langle DF(x),DG(x)\rangle_{H_x}\,\dd\mu(x),
 \qquad F,G\in\Dom(D).
\label{eq:response-form}
\end{equation}
Then $(\cE,\Dom(D))$ is a densely defined closed symmetric Markov form on
$L^2(\mu)$.
\end{corollary}

\begin{proof}
Since $D$ is closed,
\[
 \|F\|_{L^2(\mu)}^2+\cE(F,F)
 =\|F\|_{L^2(\mu)}^2+\|DF\|_{L^2_\mu(H)}^2
\]
is a complete norm on $\Dom(D)$.  The domain is dense by
\cref{prop:closure}, while symmetry and nonnegativity follow directly from
\eqref{eq:response-form}.

It remains to prove that $\cE$ is Markovian.  Let $C$ be a normal
contraction as above.  For $T_j(t)=(-j)\vee(t\wedge j)$, choose a smooth
compactly supported probability density $\rho$ and set
\[
 \rho_\varepsilon(t):=\varepsilon^{-1}\rho(t/\varepsilon),\qquad
 C_j:=(T_j\circ C)*\rho_{\varepsilon_j}
   -\bigl((T_j\circ C)*\rho_{\varepsilon_j}\bigr)(0),
 \qquad \varepsilon_j\downarrow0.
\]
Then $C_j\in C_b^\infty(\R)$, $C_j(0)=0$, $|C_j'|\le1$,
$C_j(t)\to C(t)$ for every $t$, and $|C_j(t)|\le |t|$.  \Cref{prop:closed-calculus} gives
\[
 \cE(C_j(F),C_j(F))
 =\int_X|C_j'(F)|^2\|DF\|_{H_x}^2\,\dd\mu
 \le \cE(F,F).
\]
Dominated convergence gives $C_j(F)\to C(F)$ in $L^2(\mu)$, and the lower
semicontinuity of the closed form $\cE$ gives
\[
 C(F)\in\Dom(D),\qquad
 \cE(C(F),C(F))\le\cE(F,F),
\]
which proves the Markov property.
\end{proof}

\medskip
\noindent\textbf{Proof of Theorem~\ref{thm:intro-image-closure}.}
\Cref{prop:log-derivative} proves part~(i), \cref{prop:closure} proves
part~(ii), and \cref{cor:closed-energy}, using \cref{prop:closed-calculus},
proves part~(iii).

\section{The LQG metric}
\label{sec:lqg}

We first derive the response pathwise from Weyl-reweighted geodesics.  We then
reconstruct the response measurably from the projective metric, close the
resulting image-law calculus, and transfer Gaussian energy-image-density
absolute continuity to it.  The remaining subsections establish forest
nondegeneracy and prove the joint-density theorem.  Each stage is separated
below so that the pathwise, measure-theoretic, and geometric inputs remain
distinct.

\subsection{Pathwise Weyl response and geodesic occupations}

Let $M$ be a set.  Throughout this subsection, a path means a continuous map
from a compact interval into $M$.

\begin{definition*}
For a metric $G$ on $M$ and a path $P:[a,b]\to M$, define
\[
 \len(P;G):=
 \sup_{a=t_0<\cdots<t_n=b}
 \sum_{i=1}^n G\bigl(P(t_{i-1}),P(t_i)\bigr).
\]
We call $P$ $G$-rectifiable if $\len(P;G)<\infty$.  For $z\in M$ and $r>0$,
write
\[
 B_G(z,r):=\{w\in M:G(z,w)<r\},
 \qquad
 \overline B_G(z,r):=\{w\in M:G(z,w)\le r\}.
\]
We call $G$ proper if $\overline B_G(z,r)$ is compact for every $z\in M$ and
$r>0$.
\end{definition*}

\begin{definition*}
For $x,y\in M$, we call $\gamma\in C([0,1],M)$ a $G$-geodesic from $x$ to
$y$ if
\[
 \gamma(0)=x,\qquad \gamma(1)=y,\qquad
 G(\gamma(s),\gamma(t))=G(x,y)|s-t|,
 \quad s,t\in[0,1].
\]
We call $G$ geodesic if every $x,y\in M$ can be joined by a $G$-geodesic.
\end{definition*}

Let $D$ be a proper geodesic metric on $M$.  Write $C_b(M)$ for the bounded
continuous real-valued functions on $M$ and $\mathcal B(M)$ for the Borel
$\sigma$-field, and fix $f\in C_b(M)$.
For a $D$-rectifiable path $P$, put
\[
 L_P:=\len(P;D).
\]
If $P$ is nonconstant, let
$\widehat P:[0,L_P]\to M$ be its $D$-arclength parametrization and define
\[
 \pi_P(A):=\frac1{L_P}\int_0^{L_P}\mathbf1_A(\widehat P(s))\,\dd s,
 \qquad A\in\mathcal B(M).
\]
For $t\in\R$, set $\ell_t^f(P):=0$ when $P$ is constant, and otherwise define
\[
 \ell_t^f(P):=\int_0^{L_P}e^{tf(\widehat P(s))}\,\dd s.
\]
For $x,y\in M$, define the weighted path distance
\begin{equation}
 D_t^f(x,y):=\inf_{P:x\to y}\ell_t^f(P),
\label{eq:weyl-metric}
\end{equation}
where the infimum is over $D$-rectifiable paths from $x$ to $y$.  If a
nonconstant $P$ is parameterized at constant $D$-speed on $[0,1]$, then
\[
 \pi_P=P_\#(\mathcal L^1|_{[0,1]}),
 \qquad
 \ell_t^f(P)=L_P\int_0^1e^{tf(P(u))}\,\dd u.
\]

\begin{lemma}
\label{lem:weighted-length-identity}
With the preceding definitions, for every $t\in\R$ the following hold.
\begin{enumerate}[label=\textnormal{(\roman*)}]
\item $D_t^f$ is a metric on $M$ and
\[
 e^{-|t|\|f\|_\infty}D
 \le D_t^f
 \le e^{|t|\|f\|_\infty}D.
\]
\item For every nonconstant $D$-rectifiable path $P$,
\[
 \len(P;D_t^f)=\ell_t^f(P).
\]
Consequently, $D_t^f$ is a length metric.
\item $D_t^f$ is proper and geodesic.  Moreover, if $x\ne y$ are points of
$M$ and $P$ is a $D_t^f$-geodesic from $x$ to $y$, then $P$ is
$D$-rectifiable and
\[
 D_t^f(x,y)=\len(P;D_t^f)=\ell_t^f(P).
\]
\end{enumerate}
\end{lemma}

\begin{proof}
Fix $t\in\R$ and set
\[
 w_t:=e^{tf},\qquad
 a_t:=e^{-|t|\|f\|_\infty},\qquad
 b_t:=e^{|t|\|f\|_\infty}.
\]
\smallskip
\noindent\emph{Step 1: comparison of the two metrics.}
Then $a_t\le w_t\le b_t$, so every $D$-rectifiable path $Q$ satisfies
\[
 a_t\len(Q;D)
 \le \ell_t^f(Q)
 \le b_t\len(Q;D).
\]
Taking the infimum over paths, and using a $D$-geodesic for the upper
bound, gives
\[
 a_tD(z,z')\le D_t^f(z,z')\le b_tD(z,z'),
 \qquad z,z'\in M.
\]
Together with reversal and concatenation of paths, this shows that
$D_t^f$ is a metric bi-Lipschitz equivalent to $D$, proving (i).

\smallskip
\noindent\emph{Step 2: local metric density and the length identity.}
We next identify the local metric density.  Fix $z\in M$ and
$0<\varepsilon<w_t(z)$.  By continuity of $w_t$, choose $r>0$ such that
\[
 |w_t(u)-w_t(z)|<\varepsilon
 \qquad\bigl(u\in B_D(z,r)\bigr).
\]
Suppose that
\[
 0<D(z,z')
 <\min\left\{r,\frac{a_tr}{w_t(z)-\varepsilon}\right\}.
\]
If a $D$-rectifiable path $Q$ from $z$ to $z'$ is contained in
$B_D(z,r)$, then
\[
 \ell_t^f(Q)
 \ge \bigl(w_t(z)-\varepsilon\bigr)\len(Q;D)
 \ge \bigl(w_t(z)-\varepsilon\bigr)D(z,z').
\]
If $Q$ leaves $B_D(z,r)$, its initial segment up to the first exit has
$D$-length at least $r$, and therefore
\[
 \ell_t^f(Q)
 \ge a_tr
 >\bigl(w_t(z)-\varepsilon\bigr)D(z,z').
\]
Taking the infimum over $Q$, and testing the reverse inequality with a
$D$-geodesic from $z$ to $z'$, whose image is contained in $B_D(z,r)$,
yields
\[
 \bigl(w_t(z)-\varepsilon\bigr)D(z,z')
 \le D_t^f(z,z')
 \le \bigl(w_t(z)+\varepsilon\bigr)D(z,z').
\]
Consequently,
\[
 \lim_{\substack{z'\to z\\ z'\ne z}}
 \frac{D_t^f(z,z')}{D(z,z')}
 =w_t(z).
\]

Fix a nonconstant $D$-rectifiable path $P$ and let
$\widehat P:[0,L_P]\to M$ be its $D$-arclength parametrization.  For a metric
$G$, write
\[
 |\dot{\widehat P}|_G(s)
 :=\lim_{u\to s}
 \frac{G(\widehat P(u),\widehat P(s))}{|u-s|}
\]
whenever the limit exists.  The bi-Lipschitz comparison makes $\widehat P$
Lipschitz for both $D$ and $D_t^f$.  The metric-speed theorem
\cite[Chapter~2]{BuragoBuragoIvanov2001} gives
$|\dot{\widehat P}|_D=1$ for almost every $s\in(0,L_P)$.  At each such
$s$, the preceding local limit gives
\[
 \begin{aligned}
 |\dot{\widehat P}|_{D_t^f}(s)
 &=\lim_{u\to s}
   \frac{D_t^f(\widehat P(u),\widehat P(s))}
        {D(\widehat P(u),\widehat P(s))}
   \frac{D(\widehat P(u),\widehat P(s))}{|u-s|}\\
 &=w_t(\widehat P(s)).
 \end{aligned}
\]
Using the metric-speed formula for length, we obtain
\[
 \begin{aligned}
 \len(P;D_t^f)
 &=\int_0^{L_P}|\dot{\widehat P}|_{D_t^f}(s)\,\dd s\\
 &=\int_0^{L_P}e^{tf(\widehat P(s))}\,\dd s
 =\ell_t^f(P).
 \end{aligned}
\]

The bi-Lipschitz comparison shows that a path is $D_t^f$-rectifiable if
and only if it is $D$-rectifiable.  Hence the length identity and the
definition of $D_t^f$ give
\[
 D_t^f(x,y)
 =\inf_{P:x\to y}\len(P;D_t^f),
\]
so $D_t^f$ is a length metric.  This proves (ii).

\smallskip
\noindent\emph{Step 3: properness and existence of geodesics.}
The two metrics induce the same topology, and, for $R>0$,
\[
 \overline B_{D_t^f}(x,R)
 \subseteq \overline B_D(x,a_t^{-1}R).
\]
The set on the left is $D$-closed, while the set on the right is compact.
Thus $D_t^f$ is proper.  The Hopf--Rinow theorem for length spaces
\cite[Section~2.5]{BuragoBuragoIvanov2001} implies that $D_t^f$ is geodesic.
Finally, if $P$ is a $D_t^f$-geodesic
between distinct points, the bi-Lipschitz comparison makes $P$
$D$-rectifiable, and (ii) gives
\[
 D_t^f(x,y)=\len(P;D_t^f)=\ell_t^f(P).
\]
This proves (iii).
\end{proof}

When $f$ is fixed, write $D_t:=D_t^f$.  For $x\ne y$, write
$L:=D(x,y)$ and $\Geo_D(x,y)$ for the set of $D$-geodesics from $x$ to $y$.
Every $\gamma\in\Geo_D(x,y)$ is $L$-Lipschitz and has image in the compact set
$\overline B_D(x,L)$.  Since the defining geodesic identity is preserved under
uniform limits, the Arzel\`a--Ascoli theorem shows that $\Geo_D(x,y)$ is compact
in the uniform topology.  Moreover, if $\gamma_n\to\gamma$ uniformly, then all
the images lie in a common compact set and hence
\[
 \int_M f\,\dd\pi_{\gamma_n}
 =\int_0^1 f(\gamma_n(u))\,\dd u
 \longrightarrow
 \int_0^1 f(\gamma(u))\,\dd u
 =\int_M f\,\dd\pi_\gamma.
\]
Thus $\gamma\mapsto\int_M f\,\dd\pi_\gamma$ is continuous on
$\Geo_D(x,y)$.
Let $t_n\to0$.  For each $n$, by \cref{lem:weighted-length-identity}(iii),
choose a $D_{t_n}$-geodesic from $x$ to $y$ and reparameterize it at constant
$D$-speed on $[0,1]$; denote the resulting path by $P_n$ and set
\[
 L_n:=\len(P_n;D).
\]
Since metric length is invariant under reparameterization,
\cref{lem:weighted-length-identity}(ii)--(iii) give
\[
 D_{t_n}(x,y)
 =\len(P_n;D_{t_n})
 =\ell_{t_n}^f(P_n).
\]

\begin{lemma}
\label{lem:weyl-compactness}
The following hold.
\begin{enumerate}[label=\textnormal{(\roman*)}]
\item $L_n\to L$.
\item $(P_n)$ is relatively compact in $C([0,1],M)$, and every uniform limit
      belongs to $\Geo_D(x,y)$.
\item If $P_{n_j}\to\gamma$ uniformly, then
\begin{equation}
 \pi_{P_{n_j}}\Longrightarrow\pi_\gamma.
\label{eq:occupation-convergence}
\end{equation}
\end{enumerate}
\end{lemma}

\begin{proof}
Put $M_f:=\|f\|_\infty$.  By
\cref{lem:weighted-length-identity}, testing the upper bound with any
$D$-geodesic from $x$ to $y$ gives
\[
 e^{-|t_n|M_f}L_n
 \le \ell_{t_n}^f(P_n)
 =D_{t_n}(x,y)
 \le e^{|t_n|M_f}L,
\]
whereas $L_n\ge L$.  Hence
\[
 L\le L_n\le e^{2|t_n|M_f}L,
\]
so $L_n\to L$.

For $s,t\in[0,1]$,
\[
 D(P_n(s),P_n(t))\le L_n|s-t|,
 \qquad
 P_n([0,1])\subset
 \overline B_D\!\left(x,\sup_nL_n\right).
\]
The ball is compact, so Arzel\`a--Ascoli gives relative compactness.  If
$P_n\to\gamma$ uniformly along a subsequence, then
$\gamma(0)=x$, $\gamma(1)=y$, and
\[
 D(\gamma(s),\gamma(t))\le L|s-t|.
\]
For $0\le s<t\le1$,
\[
\begin{aligned}
 L=D(\gamma(0),\gamma(1))
 &\le D(\gamma(0),\gamma(s))
     +D(\gamma(s),\gamma(t))
     +D(\gamma(t),\gamma(1))\\
 &\le Ls+D(\gamma(s),\gamma(t))+L(1-t),
\end{aligned}
\]
which forces $D(\gamma(s),\gamma(t))=L(t-s)$.  Thus
$\gamma\in\Geo_D(x,y)$.

Finally, all paths take values in one compact set.  Hence, for every
$g\in C_b(M)$,
\[
 \int_M g\,\dd\pi_{P_n}
 =\int_0^1g(P_n(u))\,\dd u
 \longrightarrow
 \int_0^1g(\gamma(u))\,\dd u
 =\int_M g\,\dd\pi_\gamma,
\]
which proves \eqref{eq:occupation-convergence}.
\end{proof}

The compactness of $\Geo_D(x,y)$ and the continuity above allow us to define
\begin{align}
 \mathcal R_f^+(x,y)
 &:=L\min_{\gamma\in\Geo_D(x,y)}\int_M f\,\dd\pi_\gamma,
 \label{eq:weyl-right-response}\\
 \mathcal R_f^-(x,y)
 &:=L\max_{\gamma\in\Geo_D(x,y)}\int_M f\,\dd\pi_\gamma.
 \label{eq:weyl-left-response}
\end{align}
The minimum in the right response reflects that a positive perturbation favors
geodesics with smaller $f$-average.  For a negative perturbation, the favored
geodesics have larger $f$-average, which explains the maximum in the left
response.
Set $m(t):=D_t(x,y)$.  For a real-valued function $a$ defined near $0$, set
\[
 a'_+(0):=\lim_{t\downarrow0}\frac{a(t)-a(0)}{t},
 \qquad
 a'_-(0):=\lim_{t\uparrow0}\frac{a(t)-a(0)}{t},
\]
whenever the limits exist.  In the unique-geodesic case, define
\[
 |\Geo_D(x,y)|=1
 \qquad\Longrightarrow\qquad
 \Geo_D(x,y)=\{\gamma_{x,y}\}.
\]

\begin{proposition}
\label{prop:weyl-response}
The one-sided derivatives of $m$ at $0$ exist and satisfy
\[
 m'_+(0)=\mathcal R_f^+(x,y),
 \qquad
 m'_-(0)=\mathcal R_f^-(x,y).
\]
Whenever $\gamma_{x,y}$ is defined, $m$ is differentiable at $0$ and
\begin{equation}
 \left.\frac{\dd}{\dd t}\right|_{t=0}\log D_t(x,y)
 =\int_M f\,\dd\pi_{\gamma_{x,y}}.
\label{eq:weyl-log}
\end{equation}
\end{proposition}

\begin{proof}
Put $M_f:=\|f\|_\infty$.  For $|t|\le1$,
\[
 |e^{ta}-1-ta|\le C_ft^2,
 \qquad |a|\le M_f,
 \qquad C_f:=\tfrac12M_f^2e^{M_f}.
\]
Testing \eqref{eq:weyl-metric} with $\gamma\in\Geo_D(x,y)$ gives
\[
 m(t)\le
 L+tL\int_M f\,\dd\pi_\gamma+C_fLt^2,
\]
and hence
\begin{equation}
 \limsup_{t\downarrow0}\frac{m(t)-L}{t}
 \le \mathcal R_f^+(x,y).
\label{eq:weyl-upper}
\end{equation}

For the reverse inequality, choose $t_n\downarrow0$ such that
\[
 \frac{m(t_n)-L}{t_n}
 \longrightarrow
 \liminf_{t\downarrow0}\frac{m(t)-L}{t},
\]
by \cref{lem:weighted-length-identity}(iii), choose a $D_{t_n}$-geodesic
and reparameterize it at constant $D$-speed; denote the resulting path by
$P_n$.  With $L_n:=\len(P_n;D)$, invariance of metric length under
reparameterization and \cref{lem:weighted-length-identity}(ii)--(iii) give
\[
 m(t_n)
 =D_{t_n}(x,y)
 =\ell_{t_n}^f(P_n).
\]
The same Taylor estimate therefore gives
\[
 m(t_n)
 \ge L_n+t_nL_n\int_M f\,\dd\pi_{P_n}-C_ft_n^2L_n.
\]
Since $L_n\ge L$,
\[
 \frac{m(t_n)-L}{t_n}
 \ge L_n\int_M f\,\dd\pi_{P_n}-C_ft_nL_n.
\]
By \cref{lem:weyl-compactness}, after passing to a subsequence,
$P_n\to\gamma$ uniformly for some $\gamma\in\Geo_D(x,y)$,
$L_n\to L$, and $\pi_{P_n}\Rightarrow\pi_\gamma$.  Therefore
\[
 \liminf_{t\downarrow0}\frac{m(t)-L}{t}
 \ge L\int_M f\,\dd\pi_\gamma
 \ge \mathcal R_f^+(x,y).
\]
Together with \eqref{eq:weyl-upper}, this gives
\[
 m'_+(0)=\mathcal R_f^+(x,y).
\]
For the left derivative, $D_t^f=D_{-t}^{-f}$ yields
\[
 \begin{aligned}
 m'_-(0)
 &= -\left.\frac{\dd}{\dd s}\right|_{s=0+}D_s^{-f}(x,y)\\
 &= -L\min_{\gamma\in\Geo_D(x,y)}\int_M(-f)\,\dd\pi_\gamma\\
 &= \mathcal R_f^-(x,y).
 \end{aligned}
\]

Whenever $\gamma_{x,y}$ is defined, \eqref{eq:weyl-right-response}--\eqref{eq:weyl-left-response} give
\[
 m'(0)=L\int_M f\,\dd\pi_{\gamma_{x,y}}.
\]
Since $m(0)=L>0$,
\[
 \left.\frac{\dd}{\dd t}\right|_{t=0}\log D_t(x,y)
 =\frac{m'(0)}{m(0)}
 =\int_M f\,\dd\pi_{\gamma_{x,y}},
\]
which is \eqref{eq:weyl-log}.
\end{proof}

\subsection{Measurable reconstruction of the response}

Fix \(s>1\) and set \(H_s:=H^s(\C)\).  Equip \(C(\C^2)\) with the
topology of local uniform convergence; with this topology it is Polish.  The
purpose of this subsection is to reconstruct the pathwise response measurably
from the projective distance ratios.  This will supply the response sections
needed for the image-law calculus.  Choose
\(\rho\in C_c^\infty(\C)\) with \(\int_\C\rho\,\dd z=1\), and for
\([u]\in\mathcal D'(\C)/\R\) write
\[
 u^\rho:=u-\langle u,\rho\rangle\mathbf 1,
\]
the representative satisfying \(\langle u^\rho,\rho\rangle=0\).  Define
\[
 \widehat D_h:=\frac{D_{h^\rho}}{D_{h^\rho}(e_0)}.
\]

When the underlying metric needs to be specified, write
\[
 \pi_\gamma^G:=\gamma_\#\bigl(\mathcal L^1|_{[0,1]}\bigr),
 \qquad \gamma\in\Geo_G(x,y),
\]
for the normalized occupation measure.

\begin{lemma}
\label{lem:normalized-metric-measurable}
The normalized metric \(\widehat D_h\) has a measurable
\(C(\C^2)\)-valued version.  For every representative
\(\widetilde h=h^\rho+c\), almost surely,
\[
 \widehat D_h
 =\frac{D_{\widetilde h}}{D_{\widetilde h}(e_0)}.
\]
Moreover, for every \(x\ne y\),
\[
 \Geo_{\widehat D_h}(x,y)=\Geo_{D_{\widetilde h}}(x,y),
 \qquad
 \pi_\gamma^{\widehat D_h}=\pi_\gamma^{D_{\widetilde h}}
 \quad
 \bigl(\gamma\in\Geo_{D_{\widetilde h}}(x,y)\bigr).
\]
\end{lemma}

\begin{proof}
Measurability follows from the composition
\[
 [h]\longmapsto h^\rho
 \longmapsto D_{h^\rho}
 \longmapsto \frac{D_{h^\rho}}{D_{h^\rho}(e_0)},
\]
where the first map is continuous and the second is measurable
\cite{GwynneMiller2021}.  For \(\widetilde h=h^\rho+c\), Weyl scaling gives
\[
 \frac{D_{\widetilde h}}{D_{\widetilde h}(e_0)}
 =\frac{e^{\xi c}D_{h^\rho}}
        {e^{\xi c}D_{h^\rho}(e_0)}
 =\widehat D_h.
\]
Thus \(\widehat D_h=aD_{\widetilde h}\) with
\(a=D_{\widetilde h}(e_0)^{-1}>0\).  Hence
\[
 \widehat D_h(\gamma(s),\gamma(t))
 =aD_{\widetilde h}(\gamma(s),\gamma(t)),
 \qquad
 \widehat D_h(x,y)=aD_{\widetilde h}(x,y),
\]
which gives
\[
 \Geo_{\widehat D_h}(x,y)=\Geo_{D_{\widetilde h}}(x,y).
\]
For every common geodesic \(\gamma\),
\[
 \pi_\gamma^{\widehat D_h}
 =\gamma_\#\bigl(\mathcal L^1|_{[0,1]}\bigr)
 =\pi_\gamma^{D_{\widetilde h}},
\]
which proves the last assertion.
\end{proof}

For each $e\in I$, fix once and for all an ordering $(x_e,y_e)$ of its two
endpoints.  This choice is used only to orient parametrized geodesics and does not
affect their normalized occupation measures.

By the basic LQG metric properties
\cite{DubedatFalconetGwynnePfefferSun2020,GwynneMiller2021} and fixed-pair
geodesic uniqueness \cite[Theorem~1.2]{MillerQian2020}, countability of $I$
gives an event
$\Omega_{\mathrm{LQG}}\in\cF$ with $\Prob(\Omega_{\mathrm{LQG}})=1$ such that,
for every $h\in\Omega_{\mathrm{LQG}}$, the space $(\C,D_h)$ is proper and
geodesic, $D_h$ induces the Euclidean topology,
\[
 D_{h+f}=e^{\xi f}\cdot D_h\qquad(f\in C(\C)),
\]
and every $e\in I$ has a unique $D_h$-geodesic between its endpoints.  Let
$\gamma_e^h$ be the unique one oriented from $x_e$ to $y_e$, so that
\[
 \Geo_{D_h}(x_e,y_e)=\{\gamma_e^h\}.
\]
Define the normalized occupation measure on all of $\Omega$ by
\begin{equation}
 \pi_e(h):=
 \begin{cases}
 (\gamma_e^h)_\#\bigl(\mathcal L^1|_{[0,1]}\bigr),
   & h\in\Omega_{\mathrm{LQG}},\\
 \delta_0,&h\notin\Omega_{\mathrm{LQG}}.
 \end{cases}
\label{eq:lqg-occupation}
\end{equation}
At this stage, $\pi_e(h)$ is used only as a pathwise object.  Its measurable
realization from the projective distance-ratio data is constructed in the
proof of \cref{prop:lqg-gradient-measurable}.

\begin{proposition}
\label{prop:lqg-response}
For every $h\in\Omega_{\mathrm{LQG}}$, $e\in I$, and $f\in C_b(\C)$,
\begin{equation}
 \left.\frac{\dd}{\dd t}\right|_{t=0}
 \log D_{h+tf}(e)
 =\xi\int_\C f\,\dd\pi_e(h).
\label{eq:lqg-log-response}
\end{equation}
Consequently, the distance-ratio coordinate $R_e$ defined in the Introduction satisfies
\begin{equation}
 \left.\frac{\dd}{\dd t}\right|_{t=0}R_e(h+tf)
 =\xi\int_\C f\,\dd\bigl(\pi_e(h)-\pi_{e_0}(h)\bigr).
\label{eq:lqg-projective-response}
\end{equation}
\end{proposition}

\begin{proof}
Fix $h\in\Omega_{\mathrm{LQG}}$, $e\in I$, and $f\in C_b(\C)$.  By Weyl scaling,
\[
 D_{h+tf}=e^{\xi tf}\cdot D_h.
\]
Applying \cref{prop:weyl-response} to $D_h$ with weight $\xi f$ gives
\[
 \left.\frac{\dd}{\dd t}\right|_{t=0}\log D_{h+tf}(e)
 =\xi\int_\C f\,\dd\pi_e(h),
\]
which proves \eqref{eq:lqg-log-response}.  Since
\[
 R_e(h+tf)=\log D_{h+tf}(e)-\log D_{h+tf}(e_0),
\]
subtracting the formula for $e_0$ gives \eqref{eq:lqg-projective-response}.
\end{proof}

We next realize the signed occupation measures as vectors in $H_s$.
Identify \(\C\) with \(\R^2\) and, for Schwartz functions, use
\[
 \widehat f(\zeta):=\int_{\R^2}e^{-iz\cdot\zeta}f(z)\,\dd z,
 \qquad
 f(z)=\int_{\R^2}e^{iz\cdot\zeta}\widehat f(\zeta)
 \,\frac{\dd\zeta}{(2\pi)^2},
\]
with
\[
 \|f\|_{H^s}^2
 :=\int_{\R^2}(1+|\zeta|^2)^s|\widehat f(\zeta)|^2
 \,\frac{\dd\zeta}{(2\pi)^2}.
\]
Let \(H^{-s}(\C)=H_s^*\) and let
\(\mathcal R_s:H^{-s}(\C)\to H_s\) be the Riesz isomorphism; in Fourier
variables it is multiplication by \((1+|\zeta|^2)^{-s}\).  Define
\[
 k_s(z,w):=\int_{\R^2}e^{i(z-w)\cdot\zeta}
 (1+|\zeta|^2)^{-s}\,\frac{\dd\zeta}{(2\pi)^2}.
\]
Since $s>1$, the kernel $k_s$ is bounded and continuous.

\begin{proposition}
\label{prop:measure-hilbertization}
\begin{enumerate}[label=\textnormal{(\roman*)}]
\item There is \(C_{2,s}<\infty\) such that every finite signed Radon measure
\(\tau\) on \(\C\) belongs to \(H^{-s}(\C)\), and
\begin{equation}
 \|\tau\|_{H^{-s}}
 \le C_{2,s}\|\tau\|_{\mathrm{TV}}.
\label{eq:sobolev-measure}
\end{equation}
\item For every such \(\tau\), \(\mathcal R_s\tau\) has the bounded
continuous representative
\begin{equation}
 (\mathcal R_s\tau)(z)=\int_{\C}k_s(z,w)\,\tau(\dd w).
\label{eq:bessel-potential}
\end{equation}
\item For finite signed Radon measures \(\nu,\eta\) on \(\C\),
\begin{equation}
 \langle \mathcal R_s\nu,\mathcal R_s\eta\rangle_{H_s}
 =\iint_{\C\times\C}k_s(z,w)\,\nu(\dd z)\eta(\dd w).
\label{eq:kernel-energy}
\end{equation}
\end{enumerate}
\end{proposition}

\begin{proof}
For \(f\in H_s\), Sobolev embedding gives
\[
 \left|\int_{\C}f\,\dd\tau\right|
 \le \|\tau\|_{\mathrm{TV}}\|f\|_\infty
 \le C_{2,s}\|\tau\|_{\mathrm{TV}}\|f\|_{H_s},
\]
which proves (i).  For (ii), put
\[
 u_\tau(z):=\int_{\C}k_s(z,w)\,\tau(\dd w),
 \qquad
 \widehat\tau(\zeta):=\int_{\C}e^{-iw\cdot\zeta}\,\tau(\dd w).
\]
Since \(|\widehat\tau|\le\|\tau\|_{\mathrm{TV}}\) and
\((1+|\zeta|^2)^{-s}\in L^1(\R^2)\), Fubini gives
\[
 u_\tau(z)=\int_{\R^2}e^{iz\cdot\zeta}(1+|\zeta|^2)^{-s}
 \widehat\tau(\zeta)\,\frac{\dd\zeta}{(2\pi)^2}.
\]
Hence
\[
 \|u_\tau\|_{H_s}^2
 =\int_{\R^2}(1+|\zeta|^2)^{-s}|\widehat\tau(\zeta)|^2
 \,\frac{\dd\zeta}{(2\pi)^2}<\infty,
\]
and the multiplier characterization gives \(u_\tau=\mathcal R_s\tau\).
Finally,
\[
 \begin{aligned}
 \langle\mathcal R_s\nu,\mathcal R_s\eta\rangle_{H_s}
 &=\langle\nu,\mathcal R_s\eta\rangle_{H^{-s},H_s}\\
 &=\int_{\C}\!\int_{\C}k_s(z,w)\,\eta(\dd w)\nu(\dd z),
 \end{aligned}
\]
and the double integral is absolutely convergent because \(k_s\) is bounded
and the measures are finite.  This proves (iii).
\end{proof}

Set
\begin{equation}
 \nu_e(h):=\pi_e(h)-\pi_{e_0}(h),
 \qquad
 \widetilde g_e(h):=\xi \mathcal R_s\nu_e(h).
\label{eq:lqg-gradient}
\end{equation}
By \eqref{eq:sobolev-measure},
\begin{equation}
 \|\widetilde g_e(h)\|_{H_s}\le2\xi C_{2,s}.
\label{eq:lqg-gradient-bound}
\end{equation}
Also set
\[
 J(h):=(R_e(h))_{e\in I}\in\R^I,
 \qquad
 \mu_L:=J_\#\Prob.
\]
Thus \(\widetilde g_e\) is the pathwise response, while the next proposition
realizes it as a measurable function of the image variable \(J\).

\begin{proposition}
\label{prop:lqg-gradient-measurable}
For every \(e\in I\), there is a measurable map
\[
 g_e:\R^I\longrightarrow H_s
\]
such that
\[
 g_e(J(h))=\widetilde g_e(h)\quad\Prob\text{-a.s.},
 \qquad
 \|g_e(x)\|_{H_s}\le2\xi C_{2,s}
 \quad\mu_L\text{-a.e. }x.
\]
In particular, \(g_e\in L^2(\mu_L;H_s)\).
\end{proposition}

\begin{proof}
The construction has four stages.  We first reconstruct the normalized metric
from its values on the dense set \(Q\), then select a geodesic measurably, pass
from the selected path to its occupation measure, and finally apply the Riesz
map.

Define
\[
 \operatorname{Res}:C(\C^2)\longrightarrow\R^{Q^2},
 \qquad
 \operatorname{Res}(D):=(D(q,r))_{(q,r)\in Q^2}.
\]
Here \(C(\C^2)\) carries the local-uniform topology fixed above and
\(\R^{Q^2}\) the product topology.  The map \(\operatorname{Res}\) is
continuous and injective.  Hence, by the Lusin--Souslin theorem
\cite[Theorem~15.1]{Kechris1995},
\[
 A:=\operatorname{Res}(C(\C^2))\in\mathcal B(\R^{Q^2}),
 \qquad
 \operatorname{Res}^{-1}:A\longrightarrow C(\C^2)
 \text{ is measurable}.
\]
Let \(S:\R^I\to\R^{Q^2}\) be given by
\[
 S(x)_{(q,r)}:=
 \begin{cases}
 0,&q=r,\\
 e^{x_{\{q,r\}}},&q\ne r,
 \end{cases}
\]
and let \(D_*(z,w):=|z-w|\).  Define the reconstructed
\(C(\C^2)\)-valued map
\[
 D^{\mathrm{rec}}(x):=
 \begin{cases}
 \operatorname{Res}^{-1}(S(x)),&S(x)\in A,\\
 D_*,&S(x)\notin A.
 \end{cases}
\]
Then \(D^{\mathrm{rec}}:\R^I\to C(\C^2)\) is measurable and
\[
 S(J(h))=\operatorname{Res}(\widehat D_h),
 \qquad
 D^{\mathrm{rec}}(J(h))=\widehat D_h
 \quad\Prob\text{-a.s.}
\]

Fix \(e\in I\), and use the ordered endpoints \((x_e,y_e)\) fixed above.
Let
\[
 \mathsf G_e\subset C(\C^2)\times C([0,1],\C)
\]
consist of the pairs \((D,P)\) satisfying
\[
 P(0)=x_e,\qquad P(1)=y_e,\qquad
 D(P(s),P(t))=|s-t|D(x_e,y_e)
 \quad(s,t\in\mathbb Q\cap[0,1]).
\]
Since \(D\) and \(P\) are continuous,
\[
 \begin{aligned}
 &D(P(s),P(t))=|s-t|D(x_e,y_e)
 &&(s,t\in\mathbb Q\cap[0,1])\\
 &\qquad\Longleftrightarrow\qquad
 D(P(s),P(t))=|s-t|D(x_e,y_e)
 &&(s,t\in[0,1]).
 \end{aligned}
\]
Thus, when \(D\) is a metric inducing the Euclidean topology, the
corresponding fiber is precisely \(\Geo_D(x_e,y_e)\).  If \(D_n\to D\)
locally uniformly and \(P_n\to P\) uniformly, then the ranges of \(P_n\) and
\(P\) are eventually contained in a common compact subset of \(\C\), and
\[
 \sup_{s,t\in[0,1]}
 \bigl|D_n(P_n(s),P_n(t))-D(P(s),P(t))\bigr|
 \longrightarrow0.
\]
The endpoint conditions also pass to the uniform limit.  Hence
\(\mathsf G_e\) is closed.  Set
\[
 B_e:=\operatorname{proj}_1\mathsf G_e.
\]
Since \(\mathsf G_e\) is closed, \(B_e\) is analytic and therefore universally
measurable.  The Jankov--von Neumann theorem
\cite[Theorem~18.1]{Kechris1995} gives a universally measurable selector
\(P_e^0:B_e\to C([0,1],\C)\) such that
\[
 (D,P_e^0(D))\in\mathsf G_e
 \qquad(D\in B_e).
\]
Fix the deterministic path
\[
 P_e^*(u):=(1-u)x_e+uy_e,
 \qquad u\in[0,1],
\]
and define
\[
 P_e(D):=
 \begin{cases}
  P_e^0(D),&D\in B_e,\\
  P_e^*,&D\notin B_e.
 \end{cases}
\]
Then \(P_e:C(\C^2)\to C([0,1],\C)\) is universally measurable and
\[
 (D,P_e(D))\in\mathsf G_e
 \qquad(D\in B_e).
\]
Since \(\widehat D_h\) and \(D_h\) have the same geodesics and the
endpoint ordering is fixed, uniqueness gives
\[
 P_e(D^{\mathrm{rec}}(J(h)))=\gamma_e^h
 \quad\Prob\text{-a.s.}
\]

\medskip
\noindent We next pass from the selected path to its occupation measure.  The
map
\[
 C([0,1],\C)\ni P
 \longmapsto
 P_\#\bigl(\mathcal L^1|_{[0,1]}\bigr)
 \in H^{-s}(\C)
\]
is continuous.  Indeed, if \(P_n\to P\) uniformly, then
\[
\begin{aligned}
 &\left\|(P_n)_\#(\mathcal L^1|_{[0,1]})
          -P_\#(\mathcal L^1|_{[0,1]})\right\|_{H^{-s}}^2\\
 &=\int_0^1\!\int_0^1
 \bigl[k_s(P_n(u),P_n(v))-k_s(P_n(u),P(v))\\
 &\hspace{36mm}-k_s(P(u),P_n(v))+k_s(P(u),P(v))\bigr]
 \,\dd u\,\dd v
 \longrightarrow0
\end{aligned}
\]
by bounded continuity of \(k_s\).  Consequently,
\[
 x\longmapsto
 (P_e(D^{\mathrm{rec}}(x)))_\#
 \bigl(\mathcal L^1|_{[0,1]}\bigr)
\]
is universally measurable.  Hence it is measurable with respect to the
\(\mu_L\)-completion of \(\mathcal B(\R^I)\).  Since \(\R^I\) is standard
Borel and \(H^{-s}(\C)\) is a separable Hilbert space, it agrees
\(\mu_L\)-almost everywhere with a Borel map
\[
 p_e:\R^I\longrightarrow H^{-s}(\C).
\]
Thus
\[
 p_e(x)
 =(P_e(D^{\mathrm{rec}}(x)))_\#
   \bigl(\mathcal L^1|_{[0,1]}\bigr)
 \quad\mu_L\text{-a.e. }x.
\]
Since \(\mu_L=J_\#\Prob\), the reconstruction identity yields
\[
 p_e(J(h))=\pi_e(h)
 \quad\Prob\text{-a.s.}
\]

\medskip
\noindent Finally, define
\[
 g_e(x):=\xi\mathcal R_s\bigl(p_e(x)-p_{e_0}(x)\bigr).
\]
Then
\[
 g_e(J(h))
 =\xi\mathcal R_s\bigl(\pi_e(h)-\pi_{e_0}(h)\bigr)
 =\widetilde g_e(h)
 \quad\Prob\text{-a.s.}
\]
and \eqref{eq:lqg-gradient-bound}, together with \(\mu_L=J_\#\Prob\), gives
\[
 \mu_L\!\left(\|g_e\|_{H_s}>2\xi C_{2,s}\right)
 =\Prob\!\left(\|g_e(J(h))\|_{H_s}>2\xi C_{2,s}\right)=0.
\]
\end{proof}

\subsection{Closed calculus on the projective metric law}

We now place the measurable response sections in the abstract framework of
\cref{sec:intro}.  The only remaining inputs are the Cameron--Martin
integration-by-parts identity and the two density conditions in
\cref{ass:intro-response}.

Let \(K=\dot H^1(\C)/\R\) be the Cameron--Martin space
\cite{Sheffield2007},
\[
 \langle f,g\rangle_K
 =\frac1{2\pi}\int_\C\nabla f\cdot\nabla g\,\dd z,
\]
so that \(H_s\hookrightarrow K\) continuously.  Let
\(W:K\to L^2(\Prob)\) be the Gaussian logarithmic derivative normalized by
\[
 \Ex[\Phi(h+tf)]
 =\Ex\!\left[\Phi(h)
   \exp\!\left\{tW(f)-\frac12t^2\|f\|_K^2\right\}\right],
 \qquad f\in K,
\]
for bounded measurable \(\Phi\).  Choose an orthonormal basis
\((f_n)_{n\ge1}\subset C_c^\infty(\C)\) of \(H_s\), and let \(V_0\) be its
rational linear span.

For \(x=(x_e)_{e\in I}\in\R^I\), write \(R_e(x):=x_e\).  We apply the
abstract calculus of \cref{sec:intro} on \((\R^I,\mu_L)\) with the constant
Hilbert space \(H_s\), taking \(r_e=R_e\) and the response sections \(g_e\)
from \cref{prop:lqg-gradient-measurable}.  Let \(\Pres_L,\Cyl_L\) and
\(F_P,Y_P,G_P\) denote the corresponding presentation, cylinder, pullback,
and response objects defined in \cref{sec:intro}.

\begin{proposition}
\label{prop:lqg-closed}
With the preceding definitions, we have:
\begin{enumerate}[label=\textnormal{(\roman*)}]
\item The assignment \(F_P\mapsto G_P\) descends to a densely defined closable
operator
\[
 D_{L,0}:\Cyl_L\subset L^2(\mu_L)\longrightarrow L^2(\mu_L;H_s).
\]
\item With \(D_L:=\overline{D_{L,0}}\),
\[
 \cE_L(F,G)
 :=\frac12\int\langle D_LF,D_LG\rangle_{H_s}\,\dd\mu_L,
 \qquad F,G\in\Dom(D_L),
\]
is a densely defined closed symmetric Markov form on \(L^2(\mu_L)\).
\end{enumerate}
\end{proposition}

\begin{proof}
We verify \cref{ass:intro-response}.

For \cref{ass:intro-response}(i), let
\(P=(\varphi;e_1,\ldots,e_n)\in\Pres_L\).  By
\cref{prop:lqg-gradient-measurable},
\[
 \|G_P(x)\|_{H_s}
 \le 2\xi C_{2,s}\sum_{i=1}^n\|\partial_i\varphi\|_\infty
 \quad\text{for }\mu_L\text{-a.e. }x.
\]
Hence
\[
 g_e\in L^2(\mu_L;H_s),\qquad
 \|f\|_{L^2(\mu_L;H_s)}=\|f\|_{H_s}\quad(f\in V_0),
\]
and
\[
 \|\partial_fF_P\|_{L^2(\mu_L)}
 \le \|f\|_{H_s}\,\|G_P\|_{L^2(\mu_L;H_s)}<\infty.
\]

For \cref{ass:intro-response}(ii), since \(I\) is countable,
\[
 \sigma(R_e:e\in I)=\mathcal B(\R^I).
\]
Thus \cref{prop:projective-density} gives
\[
 \overline{\Cyl_L}^{\,L^2(\mu_L)}=L^2(\mu_L).
\]
Combining the scalar density just proved with the density of \(V_0\) in
\(H_s\), the usual density of finite tensor products gives
\[
 \overline{\Span\{F_Pf:P\in\Pres_L,\ f\in V_0\}}
 ^{\,L^2(\mu_L;H_s)}
 =L^2(\mu_L;H_s).
\]

For \cref{ass:intro-response}(iii), fix \(f\in V_0\) and
\(P=(\varphi;e_1,\ldots,e_n)\in\Pres_L\).  By
\eqref{eq:lqg-projective-response}, the Riesz identity, and
\cref{prop:lqg-gradient-measurable},
\[
 \begin{aligned}
 \partial_fR_e(h)
 &=\xi\int_\C f\,\dd\nu_e(h)\\
 &=\langle \widetilde g_e(h),f\rangle_{H_s}\\
 &=\langle g_e(J(h)),f\rangle_{H_s}
 \end{aligned}
 \qquad\Prob\text{-a.s.}
\]
Consequently the finite-dimensional chain rule yields
\[
 \begin{aligned}
 \partial_fY_P(h)
 &=\sum_{i=1}^n
   \partial_i\varphi(R_{e_1}(h),\ldots,R_{e_n}(h))
   \langle g_{e_i}(J(h)),f\rangle_{H_s}\\
 &=\langle G_P(J(h)),f\rangle_{H_s}
 = (\partial_fF_P)(J(h))
 \end{aligned}
 \qquad\Prob\text{-a.s.}
\]
Weyl comparison also gives
\[
 |R_e(h+tf)-R_e(h)|\le2\xi|t|\|f\|_\infty,
 \qquad e\in I,
\]
and therefore
\[
 \left|\frac{Y_P(h+tf)-Y_P(h)}{t}\right|
 \le2\xi\|f\|_\infty
       \sum_{i=1}^n\|\partial_i\varphi\|_\infty,
 \qquad t\ne0,
\]
for \(\Prob\)-almost every \(h\).  The displayed bound justifies
differentiating the left-hand expectation by dominated convergence.  On the
right, \(Y_P\) is bounded and the Gaussian logarithmic derivative has finite
exponential moments, so the Cameron--Martin formula may also be differentiated
at zero.
We obtain
\[
 \begin{aligned}
 \Ex[\partial_fY_P]
 &=\left.\frac{\dd}{\dd t}\right|_{t=0}\Ex[Y_P(h+tf)]\\
 &=\left.\frac{\dd}{\dd t}\right|_{t=0}
   \Ex\!\left[Y_P(h)e^{tW(f)-t^2\|f\|_K^2/2}\right]\\
 &=\Ex[Y_PW(f)].
 \end{aligned}
\]
Thus
\[
 \beta_f:=W(f)
\]
is an ambient score.  Since \(\R^I\) is standard Borel, choose a measurable
\(b_f:\R^I\to\R\) such that
\[
 b_f(J)=\Ex[W(f)\mid\sigma(J)]
 \quad\Prob\text{-a.s.}
\]
Then
\[
 \|b_f\|_{L^2(\mu_L)}^2
 \le \Ex[|W(f)|^2]
 =\|f\|_K^2<\infty.
\]
Hence \cref{ass:intro-response} holds.  Applying
\cref{thm:intro-image-closure} gives the stated descent and closability, and
the response form without the factor $1/2$ is closed and Markovian.  Therefore
its positive scalar multiple $\cE_L$ is also a closed Markov form.
\end{proof}

\subsection{Ambient Gaussian realization}

We first construct the ambient Gaussian realization needed for the
absolute-continuity transfer.  Recall that \(h\) denotes the whole-plane GFF
modulo additive constants fixed in \cref{sec:intro-lqg-application}; as above,
\(h^\rho\) is its representative satisfying
\(\langle h^\rho,\rho\rangle=0\).  The construction uses weighted local
\(H^{-3}\)-norms.  The weights are chosen so that the Cameron--Martin space
embeds continuously and the GFF has finite second moment in the resulting
Hilbert space.

\begin{lemma}
\label{lem:ambient-hilbert-realization}
There is a separable Hilbert space
\[
 E\subset\mathcal D'(\C)/\R
\]
such that the canonical inclusion
\[
 \iota:K=\dot H^1(\C)/\R\longrightarrow E
\]
is continuous and injective, the inclusion
\(E\hookrightarrow\mathcal D'(\C)/\R\) is continuous, and \(h\) admits
an \(E\)-valued version satisfying
\begin{equation}
 \Ex\bigl[\|h\|_E^2\bigr]<\infty.
 \label{eq:ambient-gff-E-second-moment}
\end{equation}
\end{lemma}

\begin{proof}
We first obtain uniform local bounds for the GFF and for Cameron--Martin
directions, then choose summable weights and verify the two required
embeddings.

Choose \(\chi_m\in C_c^\infty(\C)\), \(m\ge1\), such that
\[
 0\le\chi_m\le1,\qquad
 \chi_m\equiv1\text{ on }B_m(0)\cup\supp\rho,
 \qquad
 \chi_{m+1}\equiv1\text{ on }\supp\chi_m.
\]
For a distribution class \([u]\in\mathcal D'(\C)/\R\), let \(u^\rho\) be
its unique representative satisfying \(\langle u^\rho,\rho\rangle=0\).
For a compactly supported distribution \(v\), initially interpret its
\(H^{-3}\)-norm as the extended Fourier norm
\[
 \|v\|_{H^{-3}(\C)}^2
 :=\int_{\R^2}(1+|\zeta|^2)^{-3}
   |\langle v,e^{-iz\cdot\zeta}\rangle|^2
   \,\frac{\dd\zeta}{(2\pi)^2}\in[0,\infty].
\]
It is finite exactly when \(v\in H^{-3}(\C)\).  For
\(v=\chi_mh^\rho\), the scalar Fourier coordinates are measurable in the
underlying sample point for fixed \(\zeta\) and continuous in \(\zeta\) for
each sample point, hence jointly measurable.  Thus this formula defines a
measurable extended random variable.
Set
\begin{equation}
 A_m:=\Ex\!\left[\|\chi_m h^\rho\|_{H^{-3}(\C)}^2\right],
 \qquad
 B_m:=\sup_{\|f\|_K\le1}
       \|\chi_m f^\rho\|_{H^{-3}(\C)}^2.
 \label{eq:ambient-Am-Bm}
\end{equation}
Both quantities are finite.  Indeed, for
\[
 \psi_0^\rho:=\psi-\left(\int_\C\psi\,\dd z\right)\rho,
 \qquad \int_\C\psi_0^\rho\,\dd z=0,
\]
the zero-integral condition gives, for \(|\zeta|\le1\),
\[
 \begin{aligned}
 |\widehat{\psi_0^\rho}(\zeta)|
 &=\left|\int_\C
   \bigl(e^{-iz\cdot\zeta}-1\bigr)\psi_0^\rho(z)\,\dd z\right|\\
 &\le |\zeta|\int_\C |z|\,|\psi_0^\rho(z)|\,\dd z\\
 &\le C_{K,\rho}|\zeta|\,\|\psi\|_{H^1(\C)},
 \qquad \supp\psi\subset K.
 \end{aligned}
\]
Here \(\psi_0^\rho\) has support in a fixed compact set and
\(\|\psi_0^\rho\|_{L^1}\le C_{K,\rho}\|\psi\|_{H^1}\).  Hence
\[
 \int_{|\zeta|\le1}
 \frac{|\widehat{\psi_0^\rho}(\zeta)|^2}{|\zeta|^2}
 \,\frac{\dd\zeta}{(2\pi)^2}
 \le C_{K,\rho}\|\psi\|_{H^1(\C)}^2.
\]
For \(|\zeta|\ge1\), Plancherel gives
\[
 \begin{aligned}
 \int_{|\zeta|\ge1}
 \frac{|\widehat{\psi_0^\rho}(\zeta)|^2}{|\zeta|^2}
 \,\frac{\dd\zeta}{(2\pi)^2}
 &\le \|\psi_0^\rho\|_{L^2(\C)}^2\\
 &\le C_{K,\rho}\|\psi\|_{H^1(\C)}^2.
 \end{aligned}
\]
Combining the low- and high-frequency estimates with the covariance of the
whole-plane GFF \cite{Sheffield2007} gives
\begin{equation}
 \begin{aligned}
 \Ex\bigl[|\langle h^\rho,\psi\rangle|^2\bigr]
 &=2\pi\int_{\R^2}
   \frac{|\widehat{\psi_0^\rho}(\zeta)|^2}{|\zeta|^2}
   \,\frac{\dd\zeta}{(2\pi)^2}\\
 &\le C_{K,\rho}\|\psi\|_{H^1(\C)}^2,
 \qquad \supp\psi\subset K,
 \end{aligned}
 \label{eq:ambient-local-covariance-bound}
\end{equation}
for every compact \(K\subset\C\).  The same estimate holds for complex-valued
tests by complexification.  Moreover,
\[
 \|\chi_m e^{-iz\cdot\zeta}\|_{H^1(\C)}^2
 \le C_m(1+|\zeta|^2).
\]
Consequently, Tonelli's theorem gives
\[
 \begin{aligned}
 A_m
 &=\int_{\R^2}(1+|\zeta|^2)^{-3}
   \Ex\bigl[|\langle h^\rho,\chi_m e^{-iz\cdot\zeta}\rangle|^2\bigr]
   \,\frac{\dd\zeta}{(2\pi)^2}\\
 &\le C_m\int_{\R^2}(1+|\zeta|^2)^{-2}
   \,\frac{\dd\zeta}{(2\pi)^2}<\infty.
 \end{aligned}
\]
On the other hand, the local Poincar\'e inequality for the representative
\(f^\rho\) gives
\[
 \|\chi_m f^\rho\|_{H^{-3}(\C)}
 \le \|\chi_m f^\rho\|_{L^2(\C)}
 \le C_m\|f\|_K,
\]
so \(B_m<\infty\).

Define
\begin{equation}
 c_m:=\frac{2^{-m}}{1+A_m+B_m}>0
 \label{eq:ambient-cm-weights}
\end{equation}
and
\begin{equation}
 \begin{aligned}
 E
 &:=\left\{[u]\in\mathcal D'(\C)/\R:
      \|[u]\|_E^2<\infty\right\},\\
 \|[u]\|_E^2
 &:=\sum_{m\ge1}c_m
      \|\chi_m u^\rho\|_{H^{-3}(\C)}^2.
 \end{aligned}
 \label{eq:ambient-E-definition}
\end{equation}
The map
\[
 \mathcal I_E:E\longrightarrow\bigoplus_{m\ge1}H^{-3}(\C),
 \qquad
 \mathcal I_E([u])
 :=\bigl(c_m^{1/2}\chi_m u^\rho\bigr)_{m\ge1},
\]
is an isometry.  Its range is the closed subspace determined by
\[
 \supp v_m\subseteq\supp\chi_m,
 \qquad
 \chi_m\bigl(c_{m+1}^{-1/2}v_{m+1}\bigr)
 =c_m^{-1/2}v_m,
 \qquad
 \left\langle c_m^{-1/2}v_m,\rho\right\rangle=0,
 \qquad m\ge1.
\]
Indeed, these relations are closed in the Hilbert direct sum.  Conversely,
suppose that \(v=(v_m)_{m\ge1}\) satisfies them and put
\(w_m:=c_m^{-1/2}v_m\).  If
\(\phi\in C_c^\infty(\C)\), choose \(m\) with
\(\supp\phi\subset B_m(0)\) and set
\[
 \langle u^\rho,\phi\rangle:=\langle w_m,\phi\rangle.
\]
Since \(\chi_m\equiv1\) on \(B_m(0)\) and
\(\chi_mw_{m+1}=w_m\), this definition is independent of \(m\).
It defines \(u^\rho\in\mathcal D'(\C)\); the normalization relations give
\(\langle u^\rho,\rho\rangle=0\), and the compatibility relations give
\(\chi_m u^\rho=w_m\) for every \(m\).  Thus \(v=\mathcal I_E([u])\), so the
range is closed and \(E\) is a separable Hilbert space.
Furthermore,
\begin{equation}
 \|\iota f\|_E^2
 \le\left(\sum_{m\ge1}c_mB_m\right)\|f\|_K^2
 \le\|f\|_K^2,
 \qquad f\in K,
 \label{eq:ambient-K-to-E}
\end{equation}
so \(\iota:K\to E\) is continuous.  It is injective because a class in
\(K\) whose normalized representative vanishes as a distribution is the zero
class.

To verify the second embedding, let \(\mathcal B\subset C_c^\infty(\C)\) be
bounded.  Such a family has a common compact support, and
\[
 \mathcal B_0
 :=\left\{
 \phi-\left(\int_\C\phi\,\dd z\right)\rho:
 \phi\in\mathcal B
 \right\}
\]
is again bounded in \(C_c^\infty(\C)\), has a common compact support, and
annihilates constants.  Choose \(m\) such that \(\chi_m\equiv1\) on the common
support of \(\mathcal B_0\).  Then
\begin{equation}
 \begin{aligned}
 \sup_{\phi\in\mathcal B}|\langle u^\rho,\phi\rangle|
 &=\sup_{\psi\in\mathcal B_0}|\langle [u],\psi\rangle|\\
 &\le c_m^{-1/2}\|[u]\|_E
      \sup_{\psi\in\mathcal B_0}\|\psi\|_{H^3(\C)}.
 \end{aligned}
 \label{eq:ambient-E-to-distributions}
\end{equation}
The normalization projection
\[
 u\longmapsto u-\langle u,\rho\rangle\mathbf1
\]
is continuous for the strong topology of \(\mathcal D'(\C)\).  It factors
through the quotient and identifies \(\mathcal D'(\C)/\R\) with the closed
hyperplane of \(\rho\)-normalized
distributions.  Thus \eqref{eq:ambient-E-to-distributions} proves that
\(E\hookrightarrow\mathcal D'(\C)/\R\) is continuous for the strong quotient
topology.

Finally, by \eqref{eq:ambient-Am-Bm}--\eqref{eq:ambient-cm-weights}
and Tonelli's theorem,
\begin{equation}
 \Ex\!\left[\sum_{m\ge1}c_m
 \|\chi_mh^\rho\|_{H^{-3}(\C)}^2\right]
 =\sum_{m\ge1}c_mA_m
 \le\sum_{m\ge1}2^{-m}<\infty.
 \label{eq:ambient-gff-E-moment}
\end{equation}
Hence the displayed series is finite almost surely.  To see measurability,
choose a countable dense set \(\{\psi_j\}_{j\ge1}\subset C_c^\infty(\C)\) in
\(H^3(\C)\).  The Borel \(\sigma\)-field of the separable space
\(H^{-3}(\C)\) is generated by the coordinates against the \(\psi_j\), and
\[
 \langle \chi_mh^\rho,\psi_j\rangle
 =\langle h^\rho,\chi_m\psi_j\rangle
\]
is measurable for every \(j\).  Thus the \(H^{-3}\)-valued coordinate maps
\(\chi_mh^\rho\) are measurable, so
\[
 \bigl(c_m^{1/2}\chi_mh^\rho\bigr)_{m\ge1}
\]
is a measurable random element of the Hilbert direct sum and belongs almost
surely to the closed range of \(\mathcal I_E\).  Applying the continuous
inverse of \(\mathcal I_E\) on its range gives an \(E\)-valued version of \(h\)
and proves \eqref{eq:ambient-gff-E-second-moment}.
\end{proof}

\begin{lemma}
\label{lem:ambient-gaussian-structure}
Let \(E\) and \(\iota:K\to E\) be furnished by
\cref{lem:ambient-hilbert-realization}, and set
\[
 \Prob_h:=h_\#\Prob.
\]
Then the following hold.
\begin{enumerate}[label=\textnormal{(\roman*)}]
\item \(\Prob_h\) is a centered Gaussian Radon measure on \(E\).
\item If \(\iota^*:E'\to K\) is the Hilbert adjoint, then
\begin{equation}
 \int_E\ell_1(u)\ell_2(u)\,\Prob_h(\dd u)
 =\langle\iota^*\ell_1,\iota^*\ell_2\rangle_K,
 \qquad \ell_1,\ell_2\in E'.
 \label{eq:ambient-covariance-factorization}
\end{equation}
Equivalently, the covariance operator of \(\Prob_h\) is
\[
 Q=\iota\iota^*:E'\longrightarrow E.
\]
\item The Cameron--Martin space of \(\Prob_h\) is \(\iota(K)\), with
\[
 \|\iota f\|_{H(\Prob_h)}=\|f\|_K,
 \qquad f\in K.
\]
There is a unique linear isometry
\[
 W_E:K\longrightarrow L^2(\Prob_h)
\]
such that
\begin{equation}
 W_E(\iota^*\ell)=\ell
 \quad\text{in }L^2(\Prob_h),
 \qquad \ell\in E'.
 \label{eq:ambient-linear-functional}
\end{equation}
After identifying \(f\in K\) with \(\iota f\in E\), the Cameron--Martin
formula is
\begin{equation}
 \int_E\Phi(u+tf)\,\Prob_h(\dd u)
 =\int_E\Phi(u)
   \exp\!\left\{tW_E(f)(u)-\frac12t^2\|f\|_K^2\right\}\Prob_h(\dd u)
 \label{eq:ambient-cm-formula}
\end{equation}
for bounded Borel \(\Phi\), \(f\in K\), and \(t\in\R\).
\end{enumerate}
\end{lemma}

\begin{proof}
For (i), the direct-sum realization in the proof of
\cref{lem:ambient-hilbert-realization} shows that restrictions to \(E\) of
finite-support coordinate functionals are dense in \(E'\).  Each such
coordinate functional is an \(E'\)-limit of finite linear combinations of
\[
 [u]\longmapsto\langle u^\rho,\psi\rangle,
 \qquad \psi\in C_c^\infty(\C),
\]
because \(C_c^\infty(\C)\) is dense in \(H^3(\C)\).  Their values on \(h\)
are centered Gaussian.  If \(\ell_j\to\ell\) in \(E'\), then
\[
 \Ex\bigl[|\ell_j(h)-\ell(h)|^2\bigr]
 \le \|\ell_j-\ell\|_{E'}^2\Ex\bigl[\|h\|_E^2\bigr]
 \longrightarrow0.
\]
Hence \(\ell(h)\) is centered Gaussian for every \(\ell\in E'\), so
\(\Prob_h\) is a centered Gaussian measure on \(E\).  Since \(E\) is
separable Hilbert, \(\Prob_h\) is Radon.

For (ii), first let
\[
 \ell_\psi([u]):=\langle u^\rho,\psi\rangle,
 \qquad \psi\in C_c^\infty(\C).
\]
The element \(\iota^*\ell_\psi\in K\) is characterized by
\[
 \langle \iota^*\ell_\psi,f\rangle_K
 =\ell_\psi(\iota f)
 =\langle f^\rho,\psi\rangle,
 \qquad f\in K.
\]
Therefore the defining covariance of the whole-plane GFF gives, for
\(\psi,\varphi\in C_c^\infty(\C)\),
\[
 \Ex\bigl[\ell_\psi(h)\ell_\varphi(h)\bigr]
 =\langle\iota^*\ell_\psi,\iota^*\ell_\varphi\rangle_K.
\]
The density used in part~(i), together with
\eqref{eq:ambient-gff-E-second-moment} and the continuity of
\(\iota^*:E'\to K\), extends this identity to all
\(\ell_1,\ell_2\in E'\), proving
\eqref{eq:ambient-covariance-factorization}.  By the definition of the
covariance operator, this is equivalent to \(Q=\iota\iota^*\).

For (iii), injectivity of \(\iota\) gives
\[
 \overline{\operatorname{Ran}\iota^*}^{\,K}
 =(\ker\iota)^\perp=K.
\]
By part~(ii), the rule
\[
 \iota^*\ell\longmapsto \ell
\]
is a well-defined isometry from \(\operatorname{Ran}\iota^*\) into
\(L^2(\Prob_h)\), so it extends uniquely to the linear isometry
\(W_E:K\to L^2(\Prob_h)\) in
\eqref{eq:ambient-linear-functional}.  The standard
Cameron--Martin-space characterization of a centered Gaussian measure
\cite{Bogachev1998}, applied to
\(Q=\iota\iota^*\) and the density of \(\operatorname{Ran}\iota^*\) in \(K\),
identifies the Cameron--Martin space with \(\iota(K)\), with norm transported
from \(K\).  Under this identification, \(W_E\) is the associated Gaussian
linear functional (Paley--Wiener map).  The Cameron--Martin theorem then gives
\eqref{eq:ambient-cm-formula}.
\end{proof}

Fix the space \(E\) furnished by \cref{lem:ambient-hilbert-realization}, and
identify \(K\) with its image under \(\iota\).  Using the same global Borel
LQG metric realization as above, define, for \(u\in E\) and \(e\in I\),
\begin{equation}
 R_e(u):=\log\frac{D_{u^\rho}(e)}{D_{u^\rho}(e_0)},
 \qquad
 J(u):=(R_e(u))_{e\in I}.
 \label{eq:ambient-ratio-map}
\end{equation}
The maps in \eqref{eq:ambient-ratio-map} are Borel because
\(E\hookrightarrow\mathcal D'(\C)/\R\), \([u]\mapsto u^\rho\), and the fixed
metric realization are Borel.  Here the first two maps are continuous for the
strong distribution topology, and the identity from the strong to the weak
distribution topology used for the metric realization is continuous.  These
maps agree \(\Prob_h\)-almost surely with the
versions used above, and therefore
\begin{equation}
 \mu_L=J_\#\Prob_h.
 \label{eq:ambient-image-law}
\end{equation}
Retain the orthonormal basis
\((f_n)_{n\ge1}\subset C_c^\infty(\C)\) of \(H_s\).  On the full
\(\Prob_h\)-measure set on which Weyl scaling holds simultaneously for all
continuous perturbations, constants cancel from the distance ratios and
\begin{equation}
 \begin{aligned}
 |R_e(u+af_n)-R_e(u+bf_n)|
 &\le \xi|a-b|\bigl(\sup_\C f_n-\inf_\C f_n\bigr)\\
 &\le2\xi|a-b|\,\|f_n\|_\infty,
 \qquad e\in I,\quad a,b\in\R.
 \end{aligned}
 \label{eq:ambient-line-Lipschitz}
\end{equation}
Pulling back \eqref{eq:ambient-cm-formula} by \(h\) gives the
Cameron--Martin identity used in \cref{prop:lqg-closed}; uniqueness of the
Gaussian logarithmic derivative therefore identifies \(W_E(f)\circ h\) with
the score \(W(f)\) used there.  We henceforth also write \(W(f)\) for
\(W_E(f)\) on \((E,\Prob_h)\).

\subsection{Energy-image-density on the ambient space}

We next introduce the directional Sobolev structure associated with the basis
\((f_n)_{n\ge1}\).  The goal is to obtain energy-image-density absolute
continuity on the ambient Gaussian space before transferring it to the
projective metric law.

For \(0\ne f\in E\), let \(AC_f\) denote the \(L^2(\Prob_h)\)-classes
admitting a Borel representative locally absolutely continuous on every
affine line parallel to \(f\).  If \(0\ne f\in K\), \(U\in AC_f\), and
\(\widetilde U\) is such a representative, set
\[
 \partial_fU
 :=\left[
 u\longmapsto
 \left.\frac{\dd}{\dd t}\right|_{t=0}\widetilde U(u+tf)
 \right]_{L^0(\Prob_h)}.
\]
For $0\ne f\in K$, Cameron--Martin quasi-invariance and Fubini's theorem show that
this derivative exists \(\Prob_h\)-almost everywhere and that its
\(L^0(\Prob_h)\)-class is independent of the chosen representative.  Let
\(\cD_s\) consist of the \(U\in L^2(\Prob_h)\) such that
\[
 U\in\bigcap_{n\ge1}AC_{f_n},
 \qquad
 \sum_{n\ge1}|\partial_{f_n}U|^2\in L^1(\Prob_h).
\]
For \(U,Z\in\cD_s\), define
\[
 \nabla_sU:=\sum_{n\ge1}\partial_{f_n}U\,f_n,
 \qquad
 \Gamma_s(U,Z):=\langle\nabla_sU,\nabla_sZ\rangle_{H_s},
 \qquad
 \cE_s(U,Z):=\frac12\int_E\Gamma_s(U,Z)\,\dd\Prob_h.
\]
Then
\[
 \|\nabla_sU\|_{L^2(\Prob_h;H_s)}^2
 =\int_E\sum_{n\ge1}|\partial_{f_n}U|^2\,\dd\Prob_h<\infty.
\]
For \(F=(F^1,\ldots,F^m)\in(\cD_s)^m\), write
\[
 \Gamma_s[F]:=\bigl(\Gamma_s(F^i,F^j)\bigr)_{i,j=1}^m.
\]

\begin{lemma}
\label{lem:ambient-directional-eid}
With the preceding definitions, the following hold.
\begin{enumerate}[label=\textnormal{(\roman*)}]
\item The operator
\[
 \nabla_s:\cD_s\subset L^2(\Prob_h)
 \longrightarrow L^2(\Prob_h;H_s)
\]
is closed.
\item For every \(m\ge1\) and \(F\in(\cD_s)^m\),
\begin{equation}
 F_\#\bigl(\det\Gamma_s[F]\cdot\Prob_h\bigr)\ll\mathcal L^m.
\label{eq:ambient-eid}
\end{equation}
\end{enumerate}
\end{lemma}

\begin{proof}
The cited criterion has three inputs in the present setting: admissibility of
the directions, strict positivity of the corresponding line densities on
bounded line segments, and square summability against every continuous linear
functional.  We verify these inputs in this order.

By the Cameron--Martin formula, every \(0\ne f_n\in H_s\subset K\) is
admissible in the sense of
\cite{BouleauHirsch1986AbsoluteContinuity}:
\[
 \Prob_h(\,\cdot-tf_n)\sim\Prob_h,
 \qquad t\in\R.
\]
By \cite[Proposition~5 and Theorem~3]{BouleauHirsch1986AbsoluteContinuity}, it
is enough to verify, for every \(n\), that
\[
 \sigma_{f_n}(A):=\int_\R\Prob_h(A-tf_n)\,\dd t,
 \qquad A\in\mathcal B(E),
\]
is equivalent to \(\Prob_h\) with a positive density
\(k_n=\dd\Prob_h/\dd\sigma_{f_n}\) satisfying
\[
 \inf_{|t|\le T}k_n(u+tf_n)>0
 \qquad(u\in E,\ T>0),
\]
and that
\[
 \sum_{n\ge1}|\ell(f_n)|^2<\infty,
 \qquad \ell\in E'.
\]
Here \(E'\) is the continuous dual of \(E\).

Fix \(0\ne f\in H_s\) and put \(a:=\|f\|_K\).  Let
\[
 \sigma_f(A):=\int_\R\Prob_h(A-tf)\,\dd t.
\]
Under the Riesz identification, the adjoint
\(\iota^*:E'\to K\) has dense range.  Choose \(\ell_j\in E'\) with
\(\iota^*\ell_j\to f\) in \(K\).  Since
\(W(\iota^*\ell)(u)=\ell(u)\) for \(\Prob_h\)-a.e. \(u\),
\[
 \|\ell_j-W(f)\|_{L^2(\Prob_h)}
 =\|\iota^*\ell_j-f\|_K\longrightarrow0.
\]
After passing to a subsequence, \(\ell_j(u)\to W(f)(u)\) for
\(\Prob_h\)-a.e. \(u\).  Hence
\[
 L_f:=\{u\in E:\lim_j\ell_j(u)\text{ exists in }\R\}
 \in\mathcal B(E),
 \qquad
 \Prob_h(L_f)=1.
\]
Since the \(\ell_j\) are linear, \(L_f\) is a vector subspace of \(E\).
Moreover,
\[
 \ell_j(f)=\langle\iota^*\ell_j,f\rangle_K\longrightarrow a^2,
\]
so \(f\in L_f\).  By Hahn--Banach choose \(\lambda\in E'\) with
\(\lambda(f)=1\), and define
\[
 w_f(u):=
 \begin{cases}
 \displaystyle
 \lim_j\ell_j\bigl(u-\lambda(u)f\bigr)+\lambda(u)a^2,
 &u-\lambda(u)f\in L_f,\\[4pt]
 \lambda(u)a^2,&u-\lambda(u)f\notin L_f.
 \end{cases}
\]
Then \(w_f\) is Borel measurable and
\[
 w_f(u+tf)=w_f(u)+ta^2,
 \qquad u\in E,\ t\in\R.
\]
On \(L_f\),
\[
 w_f(u)
 =\lim_j\ell_j(u)-\lambda(u)\lim_j\ell_j(f)+\lambda(u)a^2
 =\lim_j\ell_j(u),
\]
so \(w_f=W(f)\) \(\Prob_h\)-a.e.

For \(A\in\mathcal B(E)\), the Cameron--Martin formula and Tonelli give
\[
\begin{aligned}
 \sigma_f(A)
 &=\int_\R\Prob_h(A-tf)\,\dd t\\
 &=\int_A\!\int_\R
   \exp\!\left\{tw_f(u)-\frac12t^2a^2\right\}\dd t\,\Prob_h(\dd u)\\
 &=\int_A\frac{\sqrt{2\pi}}a
   \exp\!\left\{\frac{w_f(u)^2}{2a^2}\right\}\Prob_h(\dd u).
\end{aligned}
\]
Thus \(\sigma_f\sim\Prob_h\), and
\[
 k_f(u)=\frac a{\sqrt{2\pi}}
 \exp\!\left\{-\frac{w_f(u)^2}{2a^2}\right\}
\]
is an everywhere positive version of \(\dd\Prob_h/\dd\sigma_f\).  Hence, for
\(u\in E\) and \(T>0\),
\[
 \inf_{|t|\le T}k_f(u+tf)
 \ge\frac a{\sqrt{2\pi}}
 \exp\!\left\{-\frac{(|w_f(u)|+Ta^2)^2}{2a^2}\right\}>0.
\]
Together with Cameron--Martin admissibility, this verifies strict
admissibility of every \(f_n\).

Finally, the continuous embedding \(H_s\hookrightarrow E\) gives
\(\ell|_{H_s}\in H_s^*\) for \(\ell\in E'\), and Parseval yields
\[
 \sum_{n\ge1}|\ell(f_n)|^2
 =\|\ell|_{H_s}\|_{H_s^*}^2<\infty.
\]
With all three inputs verified, Proposition~5 gives the closed form
\((\cE_s,\cD_s)\).  Since
\[
 \|U\|_{L^2(\Prob_h)}^2+2\cE_s(U,U)
 =\|U\|_{L^2(\Prob_h)}^2
  +\|\nabla_sU\|_{L^2(\Prob_h;H_s)}^2,
\]
this is equivalent to the closedness of \(\nabla_s\), while Theorem~3 gives
\eqref{eq:ambient-eid}.
\end{proof}

We will also use the consequence of the proof that
\[
 \sigma_{f_n}\sim\Prob_h,
 \qquad n\ge1.
\]

\begin{lemma}
\label{lem:linewise-measurable-repair}
Let \(0\ne f\in E\) and let \(U:E\to\R\) be Borel with
\(U\in L^2(\Prob_h)\).  Suppose there exists \(C<\infty\) such that, for
\(\Prob_h\)-a.e. \(u\),
\[
 |U(u+tf)-U(u+sf)|\le C|t-s|,
 \qquad s,t\in\R.
\]
Then there is a Borel map
\(\widehat U:E\to\R\) such that
\begin{enumerate}[label=\textnormal{(\roman*)}]
\item \(t\mapsto\widehat U(u+tf)\) is \(C\)-Lipschitz on \(\R\) for every
      \(u\in E\);
\item for \(\Prob_h\)-a.e. \(u\),
\[
 \widehat U(u+tf)=U(u+tf),
 \qquad t\in\R.
\]
\end{enumerate}
In particular, \(\widehat U=U\) \(\Prob_h\)-a.e. and the
\(L^2(\Prob_h)\)-class of \(U\) belongs to \(AC_f\).
\end{lemma}

\begin{proof}
We construct a common Lipschitz version by choosing a measurable transversal
to the affine lines parallel to \(f\) and extending the rational restrictions
on each such line.

By Hahn--Banach choose \(\lambda\in E'\) with \(\lambda(f)=1\), and put
\[
 q(u):=u-\lambda(u)f\in\ker\lambda.
\]
For \(x\in\ker\lambda\), let \(r\mapsto U(x+rf)\) be restricted to
\(\mathbb Q\), and let \(B\subset\ker\lambda\) be the set on which this
restriction is \(C\)-Lipschitz.  Equivalently,
\[
 B=\bigl\{x\in\ker\lambda:
 |U(x+rf)-U(x+sf)|\le C|r-s|\ \text{for all }r,s\in\mathbb Q\bigr\}.
\]
The set \(B\) is a Borel subset of \(\ker\lambda\).  For \(x\in B\), denote by \(L_x:\R\to\R\) the
unique \(C\)-Lipschitz extension.  If
\[
 r_k(t):=2^{-k}\lfloor2^kt\rfloor,
\]
then
\[
 L_x(t)=\lim_{k\to\infty}U(x+r_k(t)f),
 \qquad x\in B.
\]
Define
\[
 \widehat U(u):=
 \begin{cases}
 L_{q(u)}(\lambda(u)),&q(u)\in B,\\
 0,&q(u)\notin B.
 \end{cases}
\]
On \(q^{-1}(B)\),
\[
 \widehat U(u)
 =\lim_{k\to\infty}
 U\bigl(q(u)+r_k(\lambda(u))f\bigr),
\]
so \(\widehat U\) is Borel measurable; it is zero on the Borel complement.
Since
\[
 q(u+tf)=q(u),
 \qquad
 \lambda(u+tf)=\lambda(u)+t,
\]
every affine \(f\)-line restriction of \(\widehat U\) is \(C\)-Lipschitz.
If the original restriction through \(u\) is \(C\)-Lipschitz, then
\(q(u)\in B\) and uniqueness of the Lipschitz extension gives
\[
 \widehat U(u+tf)=U(u+tf),
 \qquad t\in\R.
\]
This proves (i)--(ii).  Taking \(t=0\) in (ii) gives
\(\widehat U=U\) \(\Prob_h\)-a.e., hence the class of \(U\) lies in \(AC_f\).
\end{proof}

\subsection{Exact pullback to the projective metric law}

It remains to identify the ambient directional gradient of a pulled-back image
function with the closed gradient already constructed on the projective metric
law.  We first prove the identity on the cylinder core and then pass to the
closed domain.

Set \(W_L^{1,2}:=\Dom(D_L)\).  For \(F,G\in W_L^{1,2}\), define
\[
 \Gamma_L(F,G)(x):=\langle D_LF(x),D_LG(x)\rangle_{H_s}
 \quad\text{for }\mu_L\text{-a.e. }x.
\]

\begin{proposition}
\label{prop:lqg-exact-pullback}
With the preceding definitions, the following hold.
\begin{enumerate}[label=\textnormal{(\roman*)}]
\item For every \(F\in W_L^{1,2}\),
\[
 F(J)\in\cD_s,
 \qquad
 \nabla_s(F(J))=(D_LF)(J).
\]
\item For \(F,G\in W_L^{1,2}\),
\begin{equation}
 \Gamma_s(F(J),G(J))=\Gamma_L(F,G)(J)
 \quad\Prob_h\text{-a.e.}
\label{eq:lqg-exact-pullback}
\end{equation}
\end{enumerate}
In particular,
\begin{equation}
 \cE_s(F(J),G(J))=\cE_L(F,G),
 \qquad F,G\in W_L^{1,2}.
 \label{eq:lqg-exact-form-pullback}
\end{equation}
\end{proposition}

\begin{proof}
We begin with the cylinder core, where the pathwise response formula is
available along every basis direction after choosing suitable linewise
representatives.

First let \(P=(\varphi;e_1,\ldots,e_N)\in\Pres_L\) and put
\[
 U_P(u):=F_P(J(u)),
 \qquad u\in E.
\]
The Weyl comparison fixed above gives, for every \(n\ge1\) and
\(\Prob_h\)-almost every \(u\),
\[
 |U_P(u+af_n)-U_P(u+bf_n)|
 \le C_{P,n}|a-b|,
 \qquad a,b\in\R,
\]
where
\[
 C_{P,n}:=2\xi\|f_n\|_\infty
 \sum_{i=1}^N\|\partial_i\varphi\|_\infty.
\]
By \cref{lem:linewise-measurable-repair}, \(U_P\in AC_{f_n}\).  Let
\(\widehat U_{P,n}\) be the repaired Borel representative furnished by that
lemma.

The response identities in
\cref{prop:lqg-response,prop:measure-hilbertization,prop:lqg-gradient-measurable}
are constructed from the same global metric realization.  Therefore their
combined identity transfers to the canonical \(E\)-realization under
\(\Prob_h\).  By countability of \(I\times\mathbb N\), choose a Borel set
\(\Omega_*\subset E\) with \(\Prob_h(\Omega_*)=1\) such that, for
\(u\in\Omega_*\), \(e\in I\), and \(n\ge1\),
\[
 \left.\frac{\dd}{\dd t}\right|_{t=0}R_e(u+tf_n)
 =\langle g_e(J(u)),f_n\rangle_{H_s}.
\]
Quasi-invariance gives, for \(T>0\),
\[
 \int_E\int_{-T}^T
 \mathbf1_{\Omega_*^c}(u+tf_n)\,\dd t\,\Prob_h(\dd u)=0.
\]
Hence, for \(\Prob_h\)-a.e. base point \(u\), one has
\(u+tf_n\in\Omega_*\) for Lebesgue-a.e. \(t\).  Intersecting with the
full-measure set supplied by \cref{lem:linewise-measurable-repair}, we may also
assume
\[
 \widehat U_{P,n}(u+tf_n)=U_P(u+tf_n),
 \qquad t\in\R.
\]
For Lebesgue-a.e. \(t\), put \(v=u+tf_n\in\Omega_*\).  The ordinary
chain rule and the defining property of \(\Omega_*\) give
\[
 \begin{aligned}
 \frac{\dd}{\dd t}\widehat U_{P,n}(u+tf_n)
 &=\sum_{i=1}^N
   \partial_i\varphi\bigl(R_{e_1}(v),\ldots,R_{e_N}(v)\bigr)
   \left.\frac{\dd}{\dd r}\right|_{r=0}R_{e_i}(v+rf_n)\\
 &=\langle G_P(J(v)),f_n\rangle_{H_s}.
 \end{aligned}
\]
Thus the equality holds for
\((\Prob_h\otimes\mathcal L^1)\)-almost every \((u,t)\).  Since
\[
 ((u,t)\mapsto u+tf_n)_\#(\Prob_h\otimes\mathcal L^1)
 =\sigma_{f_n}\sim\Prob_h,
\]
the linewise derivative class therefore satisfies
\[
 \partial_{f_n}U_P
 =\langle G_P(J),f_n\rangle_{H_s}
 \quad\Prob_h\text{-a.e.}
\]
for every \(n\).  Parseval yields
\[
 \sum_{n\ge1}|\partial_{f_n}U_P|^2
 =\|G_P(J)\|_{H_s}^2\in L^1(\Prob_h),
\]
so
\[
 U_P\in\cD_s,
 \qquad
 \nabla_sU_P=G_P(J).
\]

We now pass from the cylinder core to the closed domain.  Let
\(F\in W_L^{1,2}\), and choose \(F_k\in\Cyl_L\) with
\[
 F_k\to F\quad\text{in }L^2(\mu_L),
 \qquad
 D_{L,0}F_k\to D_LF\quad\text{in }L^2(\mu_L;H_s).
\]
Since \(\mu_L=J_\#\Prob_h\),
\[
 \bigl(F_k(J),\nabla_s(F_k(J))\bigr)
 \longrightarrow
 \bigl(F(J),(D_LF)(J)\bigr)
\]
in \(L^2(\Prob_h)\times L^2(\Prob_h;H_s)\).  By
\cref{lem:ambient-directional-eid}(i),
\[
 F(J)\in\cD_s,
 \qquad
 \nabla_s(F(J))=(D_LF)(J).
\]
This proves (i), and (ii) follows by taking the \(H_s\)-inner product.
\end{proof}

For \(F=(F^1,\ldots,F^m)\in(W_L^{1,2})^m\), write
\[
 \Gamma_L[F]:=\bigl(\Gamma_L(F^i,F^j)\bigr)_{i,j=1}^m.
\]

\begin{corollary}
\label{cor:lqg-eid}
For every \(m\ge1\) and \(F\in(W_L^{1,2})^m\),
\begin{equation}
 F_\#\bigl(\det\Gamma_L[F]\cdot\mu_L\bigr)\ll\mathcal L^m.
\label{eq:lqg-eid}
\end{equation}
\end{corollary}

\begin{proof}
By \cref{prop:lqg-exact-pullback}, \(\mu_L=J_\#\Prob_h\), and
\cref{lem:ambient-directional-eid}(ii),
\[
 F_\#\bigl(\det\Gamma_L[F]\cdot\mu_L\bigr)
 =(F(J))_\#\bigl(\det\Gamma_s[F(J)]\cdot\Prob_h\bigr)
 \ll\mathcal L^m.
\]
\end{proof}

The choice \(H_s=H^s(\C)\) is analytic: the resulting response energy
depends on \(s\) and on the planar parameterization, and is not asserted to
be an intrinsic LQG covariance or a conformally invariant Dirichlet form.

\subsection{Forest nondegeneracy}

By \eqref{eq:lqg-gradient} and the injectivity of \(\mathcal R_s\), it is enough
to prove linear independence of the signed occupation measures.  We use the
following consequence of fixed-target confluence: geodesics directed away from
a fixed marked target, with their other endpoints outside a sufficiently small
metric ball, agree on a nontrivial initial segment.  By fixed-target confluence
\cite[Theorem~1.3]{GwynneMiller2020Confluence} and countability of \(Q\), fix
\(\Omega_{\mathrm{conf}}\subset\Omega_{\mathrm{LQG}}\) with
\(\Prob(\Omega_{\mathrm{conf}})=1\) on which this holds simultaneously for all
marked targets.

\begin{lemma}
\label{lem:no-marked-interior}
For every \(h\in\Omega_{\mathrm{conf}}\) and \(e\in I\),
\[
 \gamma_e^h((0,1))\cap Q=\varnothing.
\]
\end{lemma}

\begin{proof}
Fix \(h\in\Omega_{\mathrm{conf}}\) and \(e=\{a,b\}\in I\).  Orient
the unique geodesic from \(a\) to \(b\) and denote this constant-speed
parametrization by \(\gamma\).  It is either \(\gamma_e^h\) or its reversal, so
it has the same image and interior image.  Set
\[
 L:=D_h(a,b)>0.
\]
Suppose
\[
 v=\gamma(t_0)\in Q,
 \qquad 0<t_0<1.
\]
Then \(v\ne a,b\).  The two subpaths of \(\gamma\) from \(a\) and \(b\) to
\(v\) are geodesics; by uniqueness they are the marked geodesics to \(v\).
Parameterize their reversals from \(v\) at unit \(D_h\)-speed by
\[
 \eta_a(r):=\gamma\!\left(t_0-\frac rL\right),
 \quad 0\le r\le Lt_0,
 \qquad
 \eta_b(r):=\gamma\!\left(t_0+\frac rL\right),
 \quad 0\le r\le L(1-t_0).
\]
Put
\[
 s_0:=\frac L2\min\{t_0,1-t_0\}>0.
\]
Since
\[
 D_h(v,a)=Lt_0>s_0,
 \qquad
 D_h(v,b)=L(1-t_0)>s_0,
\]
we have \(a,b\notin B_{D_h}(v,s_0)\).  Fixed-target confluence at \(v\)
yields \(\delta\in(0,s_0)\) such that
\[
 \eta_a(r)=\eta_b(r)
 \quad(0\le r\le\delta).
\]
On the other hand, since \(\gamma\) is a constant-speed \(D_h\)-geodesic,
for every \(0<r<\delta\),
\[
 \begin{aligned}
 0
 &=D_h\bigl(\eta_a(r),\eta_b(r)\bigr)\\
 &=D_h\!\left(
   \gamma\!\left(t_0-\frac rL\right),
   \gamma\!\left(t_0+\frac rL\right)
   \right)\\
 &=L\left|\left(t_0+\frac rL\right)
          -\left(t_0-\frac rL\right)\right|
 =2r>0,
 \end{aligned}
\]
and we have a contradiction.
\end{proof}

Let \(S\) be a Hausdorff space, let \(N\subset S\) be finite, and let
\((N,E_F)\) be a forest.  We have the following elementary result.

\begin{lemma}
\label{lem:leaf-support}
For each \(e\in E_F\), let \(\mu_e\) be a finite positive Radon measure on
\(S\) such that
\[
 v\in\supp\mu_e
 \quad\Longleftrightarrow\quad
 v\in e,
 \qquad e\in E_F,\quad v\in N.
\]
Then \((\mu_e)_{e\in E_F}\) is linearly independent.
\end{lemma}

\begin{proof}
We induct on \(|E_F|\).  The empty case is trivial.  Suppose \(E_F\ne\varnothing\)
and
\[
 \sum_{e\in E_F}c_e\mu_e=0.
\]
Choose a leaf \(v\) in a component containing an edge, and let \(e_v\) be its
unique incident edge.  For every \(e\ne e_v\),
\(v\notin\supp\mu_e\), so there is an open neighborhood \(U_e\ni v\) with
\(\mu_e(U_e)=0\).  Put
\[
 U:=\bigcap_{e\in E_F\setminus\{e_v\}}U_e,
\]
with \(U=S\) if \(E_F=\{e_v\}\).  Then
\[
 \mu_{e_v}(U)>0,
 \qquad
 \mu_e(U)=0\quad(e\ne e_v),
\]
and hence
\[
 0=\sum_{e\in E_F}c_e\mu_e(U)
  =c_{e_v}\mu_{e_v}(U),
 \qquad
 c_{e_v}=0.
\]
The remaining edge set is again a forest, so the induction hypothesis gives
\(c_e=0\) for all \(e\in E_F\setminus\{e_v\}\).
\end{proof}

\subsection{Proof of Theorem~\ref{thm:intro-forest}}

Retain
\[
 E_F=\{e_0,e_1,\ldots,e_m\},
 \qquad
 N=\bigcup_{e\in E_F}e,
\]
from \cref{thm:intro-forest}.  For $h\in\Omega_{\mathrm{conf}}$ and
$e=\{a,b\}\in E_F$, \eqref{eq:lqg-occupation} gives
\[
 \pi_e(h)
 =(\gamma_e^h)_\#\bigl(\mathcal L^1|_{[0,1]}\bigr).
\]
Hence
\[
 \supp\pi_e(h)=\gamma_e^h([0,1]).
\]
Indeed, the right-hand side is compact and carries all of $\pi_e(h)$; if
$z=\gamma_e^h(t)$ and $U\ni z$ is open, then
$(\gamma_e^h)^{-1}(U)$ is a nonempty relatively open subset of $[0,1]$, so
\[
 \mathcal L^1\bigl((\gamma_e^h)^{-1}(U)\bigr)>0,
 \qquad
 \pi_e(h)(U)>0.
\]
Since the endpoints of $e$ belong to this support, while
\cref{lem:no-marked-interior} excludes every point of $N\setminus e$ from the
interior of $\gamma_e^h$, we obtain
\[
 v\in\supp\pi_e(h)
 \quad\Longleftrightarrow\quad
 v\in e,
 \qquad e\in E_F,\quad v\in N.
\]
Therefore \cref{lem:leaf-support} gives, for every $h\in\Omega_{\mathrm{conf}}$,
\[
 \sum_{e\in E_F}c_e\pi_e(h)=0
 \quad\Longrightarrow\quad
 c_e=0\qquad(e\in E_F).
\]

For such an $h$, suppose
\[
 \sum_{i=1}^m a_i\nu_{e_i}(h)=0
 \quad\text{in }H^{-s}(\C).
\]
Because $C_c^\infty(\C)\subset H_s$, the canonical embedding of finite signed
Radon measures into $H^{-s}(\C)$ is injective.  Using
$\nu_e=\pi_e-\pi_{e_0}$, we therefore have
\[
 \sum_{i=1}^m a_i\pi_{e_i}(h)
 -\left(\sum_{i=1}^m a_i\right)\pi_{e_0}(h)=0.
\]
The preceding linear independence forces
\[
 a_1=\cdots=a_m=0.
\]
Since $\widetilde g_e(h)=\xi\mathcal R_s\nu_e(h)$ and
$\xi\mathcal R_s$ is injective, the vectors
$\widetilde g_{e_1}(h),\ldots,\widetilde g_{e_m}(h)$ are linearly independent
for every $h\in\Omega_{\mathrm{conf}}$.  As
$\Prob(\Omega_{\mathrm{conf}})=1$ and \cref{prop:lqg-gradient-measurable} gives
\[
 g_{e_i}(J(h))=\widetilde g_{e_i}(h)
 \quad\Prob\text{-a.s.},\qquad i=1,\ldots,m,
\]
their Gram matrix is positive definite $\Prob$-almost surely.  Hence
\[
 \det\bigl(\langle g_{e_i}(J(h)),g_{e_j}(J(h))\rangle_{H_s}\bigr)_{i,j=1}^m
 >0
 \quad\Prob\text{-a.s.}
\]
Since $\mu_L=J_\#\Prob$,
\[
 \det\bigl(\langle g_{e_i}(x),g_{e_j}(x)\rangle_{H_s}\bigr)_{i,j=1}^m
 >0
 \quad\mu_L\text{-a.e. }x.
\]

The coordinate functions \(R_{e_i}\) are unbounded and hence are not
themselves in the cylinder core defined using \(C_b^\infty\).  We therefore
apply a bounded smooth change of variables before invoking the
energy-image-density statement.

Set
\[
 \mathbf R(x):=(R_{e_1}(x),\ldots,R_{e_m}(x)),
 \qquad
 T(y_1,\ldots,y_m):=(\arctan y_1,\ldots,\arctan y_m),
\]
and let $u:=T\circ\mathbf R$.  Since
$u_i=\arctan(R_{e_i})\in\Cyl_L$, \cref{prop:lqg-closed}(i) gives
\[
 D_Lu_i=\frac{g_{e_i}}{1+R_{e_i}^2},
 \qquad i=1,\ldots,m.
\]
Consequently, for $\mu_L$-almost every $x$,
\[
 \Gamma_L[u](x)
 =\left(
 \frac{\langle g_{e_i}(x),g_{e_j}(x)\rangle_{H_s}}
 {(1+R_{e_i}(x)^2)(1+R_{e_j}(x)^2)}
 \right)_{i,j=1}^m,
\]
and hence
\[
 \det\Gamma_L[u](x)
 =\left(\prod_{i=1}^m(1+R_{e_i}(x)^2)^{-2}\right)
 \det\bigl(\langle g_{e_i}(x),g_{e_j}(x)\rangle_{H_s}\bigr)_{i,j=1}^m
 >0.
\]

Let $A\in\mathcal B(( -\pi/2,\pi/2)^m)$ satisfy $\mathcal L^m(A)=0$.
By \cref{cor:lqg-eid},
\[
 0=u_\#\bigl(\det\Gamma_L[u]\cdot\mu_L\bigr)(A)
  =\int_{u^{-1}(A)}\det\Gamma_L[u] \,\dd\mu_L.
\]
Since $\det\Gamma_L[u]>0$ $\mu_L$-almost everywhere,
\[
 \mu_L\bigl(u^{-1}(A)\bigr)=0,
\]
so $u_\#\mu_L\ll\mathcal L^m$.

The map
\[
 T:\R^m\longrightarrow(-\pi/2,\pi/2)^m
\]
is a Lipschitz homeomorphism.  Hence, if
$B\in\mathcal B(\R^m)$ and $\mathcal L^m(B)=0$, then $T(B)$ is Borel and
$\mathcal L^m(T(B))=0$.  Since $u=T\circ\mathbf R$ and $T$ is injective,
\[
 \mu_L\bigl(\mathbf R^{-1}(B)\bigr)
 =\mu_L\bigl(u^{-1}(T(B))\bigr)=0.
\]
Therefore
\[
 \mathbf R_\#\mu_L\ll\mathcal L^m.
\]
Finally, $\mu_L=J_\#\Prob$ and
\[
 \mathbf R(J(h))
 =(R_{e_1}(h),\ldots,R_{e_m}(h)),
\]
so
\[
 \bigl(R_{e_1}(h),\ldots,R_{e_m}(h)\bigr)_\#\Prob
 \ll\mathcal L^m.
\]
This proves \cref{thm:intro-forest}.

The entire image-law construction and the proof of
\cref{thm:intro-forest} use no special property of the projective reference
edge beyond membership in \(I\).  Thus the construction may be repeated with
any \(e_*\in I\) as reference; in particular, \(R_e,J,\mu_L,g_e,D_L\), and
\(\Gamma_L\) are then understood relative to that choice.

\begin{corollary}
\label{cor:lqg-star-density}
The following hold.
\begin{enumerate}[label=\textnormal{(\roman*)}]
\item If $m\ge1$ and $a_0,a_1,\ldots,a_m,b\in Q$ are pairwise distinct, then
\[
\left(
\log\frac{D_h(a_1,b)}{D_h(a_0,b)},\ldots,
\log\frac{D_h(a_m,b)}{D_h(a_0,b)}
\right)_\#\Prob
\ll\mathcal L^m.
\]
\item For distinct $e,e'\in I$,
\[
\left(\log\frac{D_h(e)}{D_h(e')}\right)_\#\Prob
\ll\mathcal L^1.
\]
\end{enumerate}
\end{corollary}

\begin{proof}
For (i), repeat the construction with $e_*=\{a_0,b\}$ as reference.  The edges
\[
 e_*,\{a_1,b\},\ldots,\{a_m,b\}
\]
form a star and hence a forest, so \cref{thm:intro-forest} gives the first
claim.  For (ii), repeat it with $e_*=e'$ and $m=1$; the two distinct edges $e'$ and $e$
form a forest, and \cref{thm:intro-forest} gives the second claim.
\end{proof}


\begin{thebibliography}{99}

\bibitem{Bogachev1998}
V.~I. Bogachev,
\newblock \emph{Gaussian Measures},
\newblock Mathematical Surveys and Monographs, vol.~62,
American Mathematical Society, Providence, RI, 1998.
\bibitem{BouleauHirsch1986AbsoluteContinuity}
N.~Bouleau and F.~Hirsch,
\newblock Propriétés d'absolue continuité dans les espaces de Dirichlet
et applications aux équations différentielles stochastiques,
\newblock in \emph{Séminaire de Probabilités XX, 1984/85},
Lecture Notes in Mathematics, vol.~1204, Springer, Berlin, 1986,
pp.~131--161.
\bibitem{BuragoBuragoIvanov2001}
D.~Burago, Y.~Burago, and S.~Ivanov,
\newblock \emph{A Course in Metric Geometry},
\newblock Graduate Studies in Mathematics, vol.~33,
American Mathematical Society, Providence, RI, 2001.
\bibitem{DubedatFalconetGwynnePfefferSun2020}
J.~Dubédat, H.~Falconet, E.~Gwynne, J.~Pfeffer, and X.~Sun,
\newblock Weak LQG metrics and Liouville first passage percolation,
\newblock \emph{Probab. Theory Related Fields} \textbf{178} (2020),
369--436.
\bibitem{GwynneMiller2020Confluence}
E.~Gwynne and J.~Miller,
\newblock Confluence of geodesics in Liouville quantum gravity for
\(\gamma\in(0,2)\),
\newblock \emph{Ann. Probab.} \textbf{48} (2020), 1861--1901.
\bibitem{GwynneMiller2021}
E.~Gwynne and J.~Miller,
\newblock Existence and uniqueness of the Liouville quantum gravity
metric for \(\gamma\in(0,2)\),
\newblock \emph{Invent. Math.} \textbf{223} (2021), 213--333.
\bibitem{Kechris1995}
A.~S. Kechris,
\newblock \emph{Classical Descriptive Set Theory},
\newblock Graduate Texts in Mathematics, vol.~156,
Springer, New York, 1995.
\bibitem{MillerQian2020}
J.~Miller and W.~Qian,
\newblock The geodesics in Liouville quantum gravity are not
Schramm--Loewner evolutions,
\newblock \emph{Probab. Theory Related Fields} \textbf{177} (2020),
677--709.
\bibitem{Sheffield2007}
S.~Sheffield,
\newblock Gaussian free fields for mathematicians,
\newblock \emph{Probab. Theory Related Fields} \textbf{139} (2007),
521--541.
\end{thebibliography}
\end{document}